\documentclass[11pt]{article}
\usepackage{amssymb,amsmath} 
\usepackage{amsfonts,amsthm,mathrsfs}
\usepackage{cases} 
\usepackage{graphicx} 

\usepackage{xcolor}
\usepackage{esint} 

\newcommand{\RED}{
}
\newcommand{\BLUE}{
}

\definecolor{dGREEN}{rgb}{0.0,0.5,0.5}

\numberwithin{equation}{section}
\newtheorem{thm}{Theorem}[section]

\newtheorem{lemma}[thm]{Lemma}
\newtheorem{prop}[thm]{Proposition}

\newtheorem{definition}[thm]{Definition}
\newtheorem{remark}[thm]{Remark}

\begin{document}
\title{A few topics on total variation flows}
\author{Yoshikazu Giga
, Hirotoshi Kuroda, 
 Micha{\l} {\L}asica
}
\date{}
\maketitle
\thispagestyle{empty}

\begin{abstract}\noindent
{\BLUE Total variation gradient flows are important in several applied fields, including image analysis and materials science. 
 In this paper, we review a few basic topics including definition of a solution, explicit examples and the notion of calibrability, finite time extinction, and some regularity properties of solutions. 
 We focus on the second-order flow (possibly with weights) and the fourth-order flow. We also discuss the fractional cases. 
 }

\end{abstract}



\section{Introduction} \label{S1} 

The (essential) total variation of a function $u$ defined in a domain $\Omega$ in $\mathbb{R}^n$ is formally of the form
\[
	TV(u) = \int_\Omega |\nabla u|\; dx.
\]
Although it is a special case of $p$-Dirichlet energy
\[
	E_p(u) = \frac1p \int_\Omega |\nabla u|^p\; dx
\]
with $p=1$, there are several different features compared with $1<p<\infty$.
 This is easily observed when one considers its variation under $L^2$-metric{\BLUE, i.e., the }
 $L^2$-gradient.
 If the energy is given as
\[
	\mathcal{E}(u) = \int_\Omega f(\nabla u)\; dx,
\]
its $L^2$-gradient $\operatorname{grad}_{L^2} \mathcal{E}$ (the Euler--Lagrange operator) formally satisfies
\[
	\frac{d}{d\varepsilon} \left. \int_\Omega f\left(\nabla(u+\varepsilon\varphi)\right) dx\right|_{\varepsilon=0}
	= \int_\Omega \left(\operatorname{grad}_{L^2} \mathcal{E}(u)\right)\varphi\; dx \quad \text{for all smooth functions $\varphi$}.
\]
Neglecting the effect of the boundary, we see, by integration by parts,
\[
	\operatorname{grad}_{L^2} \mathcal{E}(u) = -\operatorname{div} \left((\nabla_q f)(\nabla u) \right)
\]
since
\begin{align*}
	\frac{d}{d\varepsilon} \left. \int_\Omega f\left(\nabla(u+\varepsilon\varphi)\right) dx\right|_{\varepsilon=0}
	&= \sum_{i=1}^n \int_\Omega \frac{\partial f}{\partial q_i}(\nabla u) \frac{\partial\varphi}{\partial x_i} (x)\; dx \\
	&= -\int_\Omega \operatorname{div} \left((\nabla_q f)(\nabla u) \right) \varphi\; dx
\end{align*}
where $\nabla_qf=(\partial f/\partial q_1,\ldots,\partial f/\partial q_n)$.
 For the $p$-Dirichlet energy, i.e., $f(q)=|q|^p/p$
\[
	\operatorname{grad}_{L^2} E_p(u) = -\operatorname{div} \left(|\nabla u|^{p-2}\nabla u \right)
\]
so that
\[
	\operatorname{grad}_{L^2} TV(u) = -\operatorname{div} \left(\nabla u/|\nabla u| \right).
\]
These operators are (degenerate) elliptic but their features depend on $p$.
 We write
\[
	\operatorname{div} \left(|\nabla u|^{p-2}\nabla u \right)
	= \sum_{1\leq i,j\leq n} a_{ij} (\nabla u) \frac{\partial^2}{\partial x_i \partial x_j} u.
\]
Then $a_{ij}(q)\to0$ as $q\to0$ if $p>2$ and all eigenvalues of $\left(a_{ij}(q)\right)$ tend to infinity as $q \to 0$ for $1<p<2$.
 In other words, $\operatorname{grad}_{L^2}E_p$ is degenerate elliptic when $p>2$ and singular for $1<p<2$.
 (In the case $p=2$, it is the Laplacian.)
 Nevertheless, the operator is elliptic in the region where $q \neq 0$.
 In the case $p=1$,
\[
	a_{ij}(q) = \frac{1}{|q|} \left( \delta_{ij}-q_i q_j/|q|^2 \right).
\]
The metric $A=\left(a_{ij}(q)\right)$ is degenerate in direction $q$ for all $q\in\mathbb{R}^n$, i.e.,
\[
	\sum_{j=1}^n a_{ij} (q) q_j = 0
	\quad\text{or}\quad Aq = 0
\]
because $|q|A$ is the orthogonal projection onto the hyperplane orthogonal to $q$.
 However, in the direction orthogonal to $q$, $A$ is singular {\BLUE at $q=0$}.
 Differently from the case $p>1$, $\operatorname{grad}_{L^2} TV$ has both singular and degenerate effects.
 From the point of energy density, $TV$ is significantly different from $E_p$ for $p>1$.
 First, the energy density $f(q)=|q|^p/p$ is non-differentiable at $q=0$ for $p=1$ while for $p>1$ it is differentiable.
 This says that the singularity of $A$ at $q=0$ is very strong compared with the case $p>1$.
 In fact, due to this singularity, the operator has nonlocal nature, as we see later.
 Second, the growth of $f(q)$ as $|q|\to\infty$ is just linear while $f(q)/|q|\to\infty$ for $p>1$.
 This means that minimizers of $TV(u)$ under suitable supplemental condition may have jump discontinuities along some hypersurface, which does not occur for $p>1$.
 Thus, problems involving $TV$ should be handled separately from problems involving $E_p$ for $p>1$.

In this paper, we intend to give a theoretical introduction to total variation flows, that is, gradient flows of $TV$.
 Typically we consider the second-order problem
\begin{equation} \label{E1TV}
	u_t + \operatorname{grad}_{L^2} TV(u) = 0
	\quad\text{\BLUE i.e.
 }\quad u_t - \operatorname{div}\left(\nabla u/|\nabla u|\right) = 0
\end{equation}
or the fourth-order problem
\begin{equation} \label{E1TV4}
	u_t + \Delta\operatorname{div}\left(\nabla u/|\nabla u|\right) = 0
\end{equation}
which is regarded as an $H^{-1}$-gradient flow, where $H^s$ denotes the Sobolev space of order $s\in\mathbb{R}$.
 We also consider the fractional case
\begin{equation} \label{E1TVS}
	u_t -(-\Delta)^s \operatorname{div}\left(\nabla u/|\nabla u|\right) = 0.
\end{equation}
As alluded before, because of the singularity at $q=0$ for the energy density $f(q)=|q|$, $\operatorname{div}\left(\nabla u/|\nabla u|\right)$ has nonlocal nature in the sense that its value at $x$ cannot be determined by the value of $u$ near $x$.
 Thus, the definition of a solution itself is nontrivial.
 Since $TV$ is still convex and lower semicontinuous in many function spaces contained in the space of Schwartz distributions, we are able to apply an abstract theory based on maximal monotone operators stated by Y.\ K\=omura and developed by H.\ Brezis and others to get a solution.
 The equation can be interpreted as a gradient flow in a Hilbert space $H$
\[
	u_t \in -\partial_H \mathcal{E}(u),
\]
where $\partial_H$ denotes the subdifferential of $\mathcal{E}$ (with respect to the metric in $H$), which is an extension of the notion of derivative.

Our first goal is to give a precise definition of a solution to total variation flows including \eqref{E1TV}, \eqref{E1TV4}, \eqref{E1TVS}.
 For the case \eqref{E1TV}, we note our formulations allow a weight, i.e., we also consider
\begin{equation} \label{E1TVW}
	bu_t = \operatorname{div}\left(a\nabla u/|\nabla u|\right).
\end{equation}
The weight is sometimes important, especially when we consider the Kobayashi--Warren--Carter system \cite{KWC} in materials science.
 As a singular limit, one has to consider even a discontinuous weight \cite{GOU}, \cite{GOSU}. 
 For \eqref{E1TV4} the interpretation as a gradient flow is sometimes non-trivial.
 In the case \eqref{E1TV4}, as discussed in \cite{GKL} the definition of a solution is quite involved for $\Omega=\mathbb{R}^n$ for $n\leq2$ because interpretation as $u_t\in-\partial_H\mathcal{E}(u)$ is not clear.
 We then calculate the subdifferential to understand the flow in a more explicit way.
 This is essential for the construction of explicit solutions.
 When we consider a characteristic function of a set as an initial datum, we wonder whether the speed $u_t$ is constant on the set and outside of the set.
 This leads to the notion of calibrability of a set.

Our second goal is to explain the notion of calibrability with special emphasis on the fourth-order problem.
 We also note that in the fourth-order problem \eqref{E1TV4}, the support of the characteristic function may move, which does not occur for the second order problem \eqref{E1TV}.
 This is because the speed may contain a delta part.
 The existence of the delta part has already been observed in \cite{Ka1} where the author discussed $TV$ perturbed by $E_p$ with some $p>1$.
 See also \cite{GG}.
 Based on calibrability, we give a few examples of explicit solutions when initial datum is a characteristic function.
 For the fourth-order problem the material is taken from \cite{GKL}, while for the second-order problem it is mostly taken from \cite{ACM}.
 As a property of a solution, we further discuss extinction time estimates as studied in \cite{GK}, \cite{GKM}.
 We discuss it for \eqref{E1TVS} which is new.
 We also present another formal argument to estimate the extinction time, which will be rigorously discussed in our forthcoming paper \cite{GKL2}.

In the last part of this paper, we review a couple of regularity results for the second-order problem. {\BLUE The main result is that the set of jump discontinuities of $u(t)$ is contained in that of $u_0$. Moreover, jump sizes are non-increasing in time. This has been first proved in \cite{CCN,CJN}. Their proof is based on an analogous result for a time-discretized problem in \cite{CCN}, which relies on the connection between the total variation and surfaces of prescribed mean curvature, given by the co-area formula, and regularity theory for those surfaces. Here, we use instead a conceptually simpler technique from \cite{CL} which does not rely on such subtle properties of $TV$. We also show a stronger estimate in the 1D case, appearing in \cite{BF, BCNO}, which implies non-expansion of jumps as well as pointwise estimate $|\nabla u(t)| \leq |\nabla u_0| $ (note that the latter does not hold in higher dimensions, see e.g.\ \cite{ACC}).} The situation is significantly different in the fourth-order case, where new jumps may be created as proved in \cite{GG}.  This behavior is also expected for \eqref{E1TVS} for $0<s<1$ as numerically suggested in \cite{GMR}.

Total variation type flows find applications in many fields including image analysis \cite{ROF}, \cite{ACM} materials science \cite{KWC} and crystal growth problems.
 For the latter two topics, see recent survey \cite{GP},  where a general second-order crystalline mean curvature flow is discussed.
 The research on fourth-order problems is less popular. 
 There is a review paper \cite{GG} which includes the development before 2010.
 {\BLUE The field is growing and we do not touch many important topics such as total variation flow of maps between Riemannian manifolds, see e.g.\ \cite{GMM, GLM, GSTU}, flows on metric spaces \cite{MST, BCP, GM}, the Wasserstein $TV$ flow \cite{BCDS, CP} or $TV$ flow with time-dependent boundary conditions \cite{BDS, GNRS}. }

 This paper is organized as follows.
 In Section \ref{S2}, we recall definitions of a solution both to the second-order and the fourth-order total variation flow.
 We also review basic unique existence results and recall the notion of a Cahn-Hoffman vector field.
 In Section \ref{S3}, we discuss the notion of calibrability with special emphasis on the fourth-order problem.
 In Section \ref{S4}, we mention several explicit solutions, mainly radially symmetric piecewise constant solutions.
 In Section \ref{S5}, we discuss several upper bounds for the extinction time.
 In Section \ref{S6}, we discuss regularity properties of the flow.

\section{Definition of a solution} \label{S2} 
\subsection{Total variation flows}\label{S2S1}

We begin with a standard definition of total variation \cite{Giu} for a locally integrable function defined in $\Omega$, where $\Omega$ is either a domain in $\mathbb{R}^n$ or flat torus $\mathbb{T}^n=\prod_{i=1}^n(\mathbb{R}/\omega_i\mathbb{Z})$ with $\omega_i>0$ for $i=1,\ldots,n$.
 For an integrable function $u$, we define the total variation of $u$ in $\Omega$ by
\[
	TV(u) := \sup \left\{ -\int_\Omega u \operatorname{div}\varphi\; dx \biggm|
	\left|\varphi(x)\right| \leq 1\ \text{for all}\ x\in\Omega,\ 
	\varphi\in \left(C_c^\infty(\Omega)\right)^n \right\},
\]
where $C_c^\infty(\Omega)$ denotes the space of all smooth functions compactly supported in $\Omega$.
 {\BLUE If $u$ belongs to the Sobolev space $W^{1,1}(\Omega)$ (that is, if its distributional derivative is an integrable function), then $TV(u) = \int_\Omega|\nabla u| dx < \infty$. More generally, $TV(u)$ is finite if and only if the distributional derivative $Du$ is a finite vector measure \cite{AFP}. In this case we say that $u$ is a \emph{function of bounded variation}, $u \in BV(\Omega)$, and we can write $TV(u) = \int_\Omega |Du|$. }
 
 For later convenience, we also define weighted total variation.
 Let $a:\Omega\to[0,\infty]$ be a lower semicontinuous function.
 We set
\[
 TV_a(u) := \sup \left\{ -\int_\Omega u \operatorname{div}\varphi\; dx \biggm|
	\left|\varphi(x)\right| \leq a(x)\ \text{for all}\ x\in\Omega,\ 
	\varphi\in \left(C_c^\infty(\Omega)\right)^n \right\}.
\]
This definition can be easily extended to the case when $u$ is a Schwartz distribution, i.e., $u\in\mathcal{D}'(\Omega)$, by replacing $-\int_\Omega u \operatorname{div}\varphi\; dx$ by a canonical pair $-\langle u,\operatorname{div}\varphi\rangle$.
 Since $u_m\to u$ in $\mathcal{D}'(\Omega)$ implies $\langle u_m,\operatorname{div}\varphi\rangle\to\langle u,\operatorname{div}\varphi\rangle$, $TV_a(u)$ is a supremum of (sequentially) continuous function in $\mathcal{D}'(\Omega)$.
 Thus, $TV_a(u)$ is lower (sequentially) semicontinuous on $\mathcal{D}'(\Omega)$.
 In particular, $TV_a(u)$ is lower semicontinuous on $L^p(\Omega)$, $p\geq1$.
 Since $TV_a(u)$ is a supremum of linear functionals $u\mapsto-\langle u,\operatorname{div}\varphi\rangle$ it must be convex on $\mathcal{D}'(\Omega)$.
 Note that our definition on $TV_a(u)$ is different from the one given by \cite{AB} where the condition $\left|\varphi(x)\right|\leq a(x)$ is only imposed for almost all $x\in\Omega$ when $a$ is discontinuous.
 In their definition, $TV_a=TV$ if $a\equiv1$ almost everywhere, while in our definition it may happen that
\[
 	TV_a(u) < TV(u)
\]
for some $u$ even if $a=1$ a.e.
 This is already observed  in \cite{AB}.
 Indeed, let us consider $\Omega=(-1,1)$ and $a(x)=1$ for $x\neq0$, $a(0)=a_0<1$ ($a_0\geq0$).
 Then, we see that
\[
	TV_a(u) = \int_{\Omega_-} \left|\frac{du}{dx}\right| + \int_{\Omega_+} \left|\frac{du}{dx}\right|
	+ a_0 \left| u(+0) - u(-0) \right| 
\]
with $\Omega_-=(-1,0)$, $\Omega_+=(0,1)$, where $u(\pm0)=\lim_{\delta\downarrow0}u(\pm\delta)$.
 Such a type of discontinuous lower semicontinuous function $a$ will be important when we study a singular limit of the Kobayashi--Warren--Carter energy as discussed in \cite{GOU} and \cite{GOSU}.
 Although the total variation is defined under Finsler type metric \cite{AB} which is important to study crystalline curvature \cite{GP}, we do not touch this problem in this paper.

We recall a classical theory for the gradient flow of a convex functional in a Hilbert space due to Y.\ K\=omura \cite{Ko} and H.\ Brezis \cite{Br}.
\begin{prop} \label{PAb} 
Let $H$ be a (real) Hilbert space.
 Let $\mathcal{E}$ be a lower semicontinuous functional on $H$ with values in $(-\infty,\infty]$ and $\mathcal{E}\not\equiv\infty$.
 Then for any $u_0\in \overline{D(\mathcal{E})}$, there exists a unique $u\in C\left([0,\infty),H\right)$ with $u_t\in\bigcap_{\delta>0}L^2\left((\delta,\infty),H\right)$ such that
\begin{equation} \label{EAb}
	u_t(t)\in -\partial_H \mathcal{E}\left(u(t)\right) \quad\text{for a.e. }t>0, \quad
	u(0) = u_0.
\end{equation}
If $u_0\in D(\mathcal{E})$, then $\delta=0$ is allowed.
\end{prop}
%
The symbol $\partial_H E$ denotes the subdifferential of $\mathcal{E}$ in $H$, i.e.,
\[
	\partial_H\mathcal{E}(v) = \left\{ f \in H \bigm|
	\mathcal{E}(v+h) - \mathcal{E}(v) \geq (h,f)_H\ \text{for all}\ h\in H \right\}
\]
for $v\in D(\mathcal{E})=\left\{f\in H\bigm|\mathcal{E}(f)<\infty\right\}$, where $(\cdot,\cdot)_H$ denotes the inner product in $H$. {\BLUE
One way to construct the solution to \eqref{EAb} is the \emph{minimizing movements scheme}. For given $f \in H$, $\lambda >0$, consider the functional $\mathcal{E}^\lambda_f \colon H \to (- \infty, \infty]$ of $\mathcal{E}$ given by 
\begin{equation} \label{min_mov}
\mathcal{E}^\lambda_f(w) = \lambda \mathcal{E}(w) + \frac{1}{2} \|w - f\|_H^2 .
\end{equation} 
Like $\mathcal{E}$, $\mathcal{E}^\lambda_f$ is lower semicontinuous and satisfies $\mathcal{E}^\lambda_f \not \equiv \infty$. Moreover, it is coercive and strictly convex, so it has a unique minimizer. For given $N \in \mathbb{N}$, we inductively produce a sequence $(u^N_k)_{k \in \mathbb{N}}$ of elements of $H$ by setting $u^N_k$ to be the minimizer of $\mathcal{E}^{1/N}_{u^N_{k-1}}$, $k \in \mathbb{N}$ and $u^N_0 = u_0$. Then, defining $u^N \in L^\infty(0, \infty;H)$ by 
\begin{equation} \label{min_mov_approx} u^N(t) = u^N_k \quad \text{for } t \in [k/N, (k+1)/N),
\end{equation} 
one can show that $u^N$ converges locally uniformly on $[0, \infty)$ to a solution to \eqref{EAb}. The solution is unique by monotonicity of the subdifferential of a convex functional. 
}

For a given non-negative measurable function $b\in L^\infty(\Omega)$ with $1/b\in L^\infty(\Omega)$, we define an inner product on $L^2(\Omega)$ by
\[
	(u,v)_{L_b^2} := \int_\Omega b(x) u(x) v(x)\; dx.
\]
This is equivalent to the standard inner product on $L^2$, $(u,v)_{L^2}$, corresponding to $b\equiv1$.
 Let $L_b^2(\Omega)$ be the space $L^2(\Omega)$ equipped with the inner product $(\ ,\ )_{L_b^2}$.
 The equation $u_t\in-\partial_{L_b^2}TV_a(u)$ formally corresponds to
 \begin{align}
\begin{aligned} \label{ETV}
	&b\frac{\partial u}{\partial t} = \operatorname{div}\left(a\frac{\nabla u}{|\nabla u|}\right) \quad\text{in}\quad  \Omega\times[0,\infty), \\
	&\frac{\partial u}{\partial\nu} = 0 \quad\text{on}\quad \partial\Omega\times(0,\infty) 
	\quad\text{if there is a boundary}\quad \partial\Omega \quad\text{of}\quad \Omega.
\end{aligned} 
\end{align} 
This type of equation is discussed in \cite{GGK} in one-dimensional periodic case, i.e., $\Omega=\mathbb{T}$.
 Here $\partial u/\partial \nu$ is the normal derivative of $u$ on $\partial\Omega$.
\begin{definition} \label{DTV}
We say that $u\in C\left([0,\infty),L_b^2(\Omega)\right)$ with $u_t\in\bigcap_{\delta>0}L^2\left((\delta,\infty),L_b^2(\Omega)\right)$ is a solution of \eqref{ETV} with initial datum $u_0\in L_b^2(\Omega)$ if it satisfies \eqref{EAb} with $H=L_b^2(\Omega)$ and $\mathcal{E}=TV_a$.
\end{definition}
We note that $D(TV_a)$ is dense in $L_b^2(\Omega)$ since $TV_a(u)$ is finite on a space $C_c^\infty(\Omega)$, which is dense in $L_b^2(\Omega)$.
 Since $TV_a$ is lower semicontinuous in $\mathcal{D}'$, it is also lower semicontinuous in $L_b^2(\Omega)$.
 Since $TV$ is convex in $L_b^2(\Omega)$, Proposition \ref{PAb} yields
\begin{thm} \label{TTV}
Let $b\in L^\infty(\Omega)$ be nonnegative with $1/b\in L^\infty(\Omega)$.
 Let $a:\Omega\to[0,\infty]$ be a lower semicontinuous function.
 Then for any $u_0\in L_b^2(\Omega)$, there exists a unique solution $u$ to \eqref{ETV}.
\end{thm}
We are interested in considering a higher order total variation flow of the form
\[
	\frac{\partial u}{\partial t} = (-\Delta) \operatorname{div}\left(a\frac{\nabla u}{|\nabla u|}\right).
\]
We first consider the case $\Omega=\mathbb{T}^n$.
 For simplicity, we set $\omega_i=1$ ($1\leq i\leq n$).
 In this case, the homogeneous Sobolev space can be defined by imposing the average free condition, using Fourier series.
 For $s\in\mathbb{R}$, we set
\[
	\dot{H}_\mathrm{av}^s(\mathbb{T}^n)
	= \left\{ u = \sum_{\substack{m\in\mathbb{Z}^n\\m\neq0}} a_m e^{2\pi ix\cdot m} \in \mathcal{D}' \Biggm|
	\|u\|_{\dot{H}_\mathrm{av}^s}^2 := \sum_{\substack{m\in\mathbb{Z}^n\\m\neq0}} |m|^{2s} |a_m|^2 < \infty \right\}.
\]
This space is a (complex) Hilbert space with inner product
\[
	(( u,v ))_s := \sum_{m\neq0} a_m \overline{b}_m |m|^{2s}, \quad
	u = \sum_{m\neq0} a_m e^{2\pi ix\cdot m}, \quad
	v = \sum_{m\neq0} b_m e^{2\pi ix\cdot m}.
\]
It is an average free space.
 We consider the space of all real-valued functions in $\dot{H}_\mathrm{av}^s$ which is still denoted by $\dot{H}_\mathrm{av}^s$.
 The total variation $TV_a$ is well-defined on any $\dot{H}_\mathrm{av}^s$ since $\dot{H}_\mathrm{av}^s$ can be viewed as a subspace of $\mathcal{D}'(\mathbb{T}^n)$.
 Since $TV_a$ is convex and lower semicontinuous, Proposition \ref{PAb} yields
\begin{thm} \label{TPV}
Let $a:\mathbb{T}^n\to[0,\infty]$ be a lower semicontinuous function.
 Let $s\in\mathbb{R}$.
 For any $u_0\in\dot{H}_\mathrm{av}^{-s}(\mathbb{T}^n)$, there is a unique $u\in C\left([0,\infty), \dot{H}_\mathrm{av}^{-s}(\mathbb{T}^n)\right)$ with $u_t\in\bigcap_{\delta>0} L^2\left((\delta,\infty), \dot{H}_\mathrm{av}^{-s}(\mathbb{T}^n)\right)$ satisfying
\[
	u_t \in - \partial_{\dot{H}_\mathrm{av}^{-s}} TV_a(u), \quad
	\text{a.e.}\quad t > 0, \quad
	u(0) = u_0.
\]
We simply say $u$ is a solution of
\begin{equation} \label{ETVP}
	u_t = (-\Delta)^s \operatorname{div}\left(a\nabla u/|\nabla u|\right)
\end{equation}
with initial data $u_0$, where $\Delta$ denotes the Laplacian.
\end{thm}
The operator $(-\Delta)^s$ comes from relation of $\partial_{\dot{H}^{-s}_{\mathrm{av}}}$ and $\partial_{L^2}$.
 Let us give a formal explanation for $s=1$.
 If $f\in\partial_{L^2}TV_a(v)$, then
\[
	TV_a(v+h) - TV_a(v) \geq (f,h)_{L^2} \quad \text{for} \quad h\in L^2(\Omega).
\] 
 Since $\dot{H}_\mathrm{av}^{-1}$ is the dual of $\dot{H}_\mathrm{av}^1$ and since $-\Delta$ is the canonical isometry from $\dot{H}_\mathrm{av}^1$ to $\dot{H}_\mathrm{av}^{-1}$, i.e.,
\[
	-\Delta: u \mapsto((u,\cdot))_1,
\]
we see
\[
	(f,h)_{L^2} = ((-\Delta f,h))_{-1}
\]
since $((u,v))_{-1}=\left( (-\Delta)^{-1}u,v \right)_{L^2}$.
 If $(-\Delta)f\in\dot{H}^{-1}_{\BLUE \mathrm{av}}$, by density of $\dot{H}_\mathrm{av}^1$ in {\BLUE $\dot{H}_\mathrm{av}^{-1}$} 
 , we conclude that $-\Delta f\in\partial_{\dot{H}^{-1}_{\BLUE \mathrm{av}}}TV_a(v)$.

If $\Omega=\mathbb{R}^n$, the definition of the fourth order total variation flow is more involved, especially for $n=1,2$.
 We consider an inner product
\[
 	((u,v))_1 := \int_{\mathbb{R}^n} \nabla u\cdot\nabla v\; dx
\]
for $u,v\in C_c^\infty(\mathbb{R}^n)$.
 Let $D_0^1(\mathbb{R}^n)$ be the completion of $C_c^\infty(\mathbb{R}^n)$ in the norm $\|u\|_1=((u,u))_1^{1/2}$.
 It is a Hilbert space equipped with $((u,v))_1$.
 For $n\geq3$, this space is identified with
\[
	D_0^1(\mathbb{R}^n) = D^1(\mathbb{R}^n) \cap L^{2^*}(\mathbb{R}^n),\quad
	D^1(\mathbb{R}^n) = \left\{ u \in L_{loc}^1(\mathbb{R}^n) \bigm|
	\nabla u \in L^2(\mathbb{R}^n) \right\},
\]
where $2^*=2n/(n-2)$ is the corresponding Sobolev exponent; i.e., the exponent such that $D_0^1(\mathbb{R}^n)\subset L^{2^*}(\mathbb{R}^n)$; see e.g.\ \cite{Gal}.
 However, for $n\leq2$, this $D_0^1(\mathbb{R}^n)$ is not a subspace of $L_{loc}^1(\mathbb{R}^n)$.
 Instead, it is isometrically identified with the quotient space $\dot{D}^1(\mathbb{R}^n):=D^1(\mathbb{R}^n)/\mathbb{R}$  equipped with the inner product $((u,v))_1$; see e.g.\ \cite{Gal}.
 We need to be careful because an element of $D_0^1(\mathbb{R}^n)$ is determined up to constant. In the case $\Omega=\mathbb{T}^n$, the space $D^1(\mathbb{T}^n)$ has a direct sum decomposition
\[
	D^1(\mathbb{T}^n) = \dot{H}_\mathrm{av}^1(\mathbb{T}^n) \oplus \mathbb{R}
\]
which corresponds to a decomposition of $u\in D^1(\mathbb{T}^n)$ as
\[
	u = (u-u_c) + u_c,
\]
where $u_c$ is the average of $u$ over $\mathbb{T}^n$.
 Thus, the space $\dot{D}(\mathbb{T}^n):=D^1(\mathbb{T}^n)/\mathbb{R}$ is identified with a subspace $\dot{H}_\mathrm{av}^1(\mathbb{T}^n)$ of $D^1(\mathbb{T}^n)$.
 In the case of $\mathbb{R}^n$, no such decomposition is available.

 We are interested in the dual space ${\RED (} D_0^1(\mathbb{R}^n) {\RED )}'$.
 Let $-\Delta$ denote the canonical isometry from $D_0^1(\mathbb{R}^n)$ and its dual, i.e.,
\[
	-\Delta : u \mapsto ((u,\cdot))_1.
\]
The inner product of ${\RED (} D_0^1(\mathbb{R}^n) {\RED )}'$ is defined as
\[
	((u,v))_{\left(D_0^1(\mathbb{R}^n)\right)'}
	:= (((-\Delta)^{-1}u, (-\Delta)^{-1}v ))_{D_0^1(\mathbb{R}^n)}.
\]
We introduce a subspace $\tilde{D}^{-1}(\mathbb{R}^n)\subset D_0^1(\mathbb{R}^n)'$ of the form
\begin{align*}
	&\tilde{D}^{-1}(\mathbb{R}^n) = \left\{ w \mapsto \int_{\mathbb{R}^n}uv\;dx \biggm|
	u\in C_c^\infty(\mathbb{R}^n) \right\} 
	\quad\text{if}\quad n \geq 3, \\
	&\tilde{D}^{-1}(\mathbb{R}^n) = \left\{ w \mapsto \int_{\mathbb{R}^n}uw\;dx \biggm|
	u\in C_{c,\mathrm{av}}^\infty(\mathbb{R}^n) \right\}
	\quad\text{if}\quad n =1 \quad\text{or}\quad n =2,
\end{align*}
where
\[
	C_{c,\mathrm{av}}^\infty(\mathbb{R}^n)
	= \left\{ u\in C_c^\infty(\mathbb{R}^n)  \biggm|
	\int_{\mathbb{R}^n} u\;dx = 0 \right\}.
\]
It is well known that $\tilde{D}^{-1}(\mathbb{R}^n)$ is dense in $\left(D_0^1(\mathbb{R}^n)\right)'$; see e.g.\ \cite{Gal}.
 In the case $n=1,2$, the restriction to the average-free space $C_{c,\mathrm{av}}^\infty(\mathbb{R}^n)$ is necessary for the functionals to be well-defined on $D^1(\mathbb{R}^n)/\mathbb{R}$.
 Since $\mathcal{D}=C_c^\infty(\mathbb{R}^n)$ is continuously embedded in $D_0^1(\mathbb{R}^n)$, $D^{-1}(\mathbb{R}^n)=\left(D_0^1(\mathbb{R}^n)\right)'$ can be viewed as a subspace of $\mathcal{D}'(\mathbb{R}^n)$.
 Thus, $TV_a(u)$ for $u\in D^{-1}(\mathbb{R}^n)$ is well-defined and it is convex, lower semicontinuous on the Hilbert space $D^{-1}(\mathbb{R}^n)$ provided that $a$ is a lower semicontinuous.
 Proposition \ref{PAb} guarantees the existence of a unique solution to $u_t\in-\partial_{D^{-1}}TV_a(u)$ with initial datum $u_0\in D^{-1}(\mathbb{R}^n)$.
 This is a rigorous way to interpret $u_t=-\Delta\operatorname{div}\left(a\nabla u/|\nabla u|\right)$.
 Unfortunately, this existence result has a drawback even if $a\equiv1$: the characteristic function $1_K$ of a set $K$ does not belong to $D^{-1}(\mathbb{R}^n)$ unless the Lebesgue measure of $K$ equals zero for $n=1,2$, since $D^{-1}(\mathbb{R}^n)$ requires a kind of average-free condition for $n=1,2$.
 In fact for $u_0\in L^2(\mathbb{R}^n)$ with compact support, $u_0\in D^{-1}(\mathbb{R}^n)$ if and only if $\int_{\mathbb{R}^n}u_0\;dx=0$ as proved in \cite[Lemma 17]{GKL}.
 We have to extend space $D^{-1}$ when $n\leq2$.
 This is quite involved; see discussion at the end of Section \ref{S2S2}.
 We refer to \cite{GKL} for details when $a\equiv1$.
 In the case where $a$ depends on $x$, the argument in \cite{GKL} still works provided that the approximation lemma \cite[Lemma 6]{GKL} can be extended to $TV_a$.

In general, we can consider the space $D_0^s$ (with $s\in\mathbb{R}$) which is the completion of $C_c^\infty(\mathbb{R}^n)$ in the norm
\[
	\|{\RED u}\|_{D_0^s}^{{\RED 2}} = \int_{\mathbb{R}^n} |\xi|^{2s} \left| \hat{u}(\xi)\right|^2 d\xi,
\]
where $\hat{u}$ denotes the Fourier transform of $u$, i.e., $\hat{u}(\xi)=\int_{\mathbb{R}^n}e^{-ix\cdot\xi}u(x)\;dx$.
 For $0<s\leq1$, the space $D_0^s\subset L_{loc}^1$ for $n>2s$ but again this does not hold for $n\leq2s$.
 In a similar way, we are able to define a total variation flow $u_t=(-\Delta)^s\operatorname{div}\left(\nabla u/|\nabla u|\right)$ for $n>2s$, whose existence is proved by Proposition \ref{PAb}.
\begin{remark} \label{RBou}
\begin{enumerate}
\item[(i)] We are able to consider higher-order problems in a domain with boundary.
 Even in the case of $s=1$, there are several choices depending on what kind of boundary condition we impose for the Laplacian.
 If we impose the homogeneous Dirichlet boundary condition, then the resulting equation is formally of the form
\begin{gather*}
	\frac{\partial u}{\partial t} = -\Delta \operatorname{div} \left(a \nabla u/|\nabla u|\right)
	\quad\text{in}\quad \Omega\times(0, \infty) \\
	{\RED u = 0}, \quad
	\operatorname{div} \left(a \nabla u/|\nabla u|\right) = 0
	\quad\text{on}\quad \partial\Omega\times(0, \infty).\end{gather*}
A rigorous formulation is given in \cite{GKM}.
 It is enough to take $H^{-1}(\Omega)=\left(H_0^1(\Omega)\right)'$ as the Hilbert space $H$.
 The Sobolev space $H_0^1(\Omega)$ can be defined similarly as $D_0^1(\mathbb{R}^n)$ by replacing $\mathbb{R}^n$ by $\Omega$.
 By the Poincar\'e inequality, this space $H_0^1(\Omega)$ belongs to $L^2(\Omega)$ so its dual $H^{-1}(\Omega)$ includes $L^2(\Omega)$.
 However, if we consider the Laplace operator with the homogeneous Neumann boundary condition, the correct choice of the space is $D_N^{-1}=\left(D^1(\Omega)/\mathbb{R}\right)'$, and expected boundary condition is
\[
	\frac{\partial u}{\partial\nu} = 0, \quad
	\frac{\partial}{\partial\nu} \operatorname{div} \left(a 		\frac{\nabla u}{|\nabla u|} \right) = 0 \quad\text{on}\quad
	\partial\Omega \times (0,\infty).
\]
Similar to the case $D^{-1}(\mathbb{R}^n)$, the analysis is quite involved.
 It will be discussed in our forthcoming paper.
\item[(i\hspace{-0.1em}i)] Even in the second-order problem, if one would like to consider the Dirichlet problem, i.e., $u=0$ on $\partial\Omega$, one should replace the energy functional $TV_a$ by
\[
	TV_a (u) + \int_{\partial\Omega} a |u|\; d\mathcal{H}^{n-1}.
\]
This type of problem is discussed in \cite{ACM} at least for $a\equiv1$. 
\end{enumerate}
\end{remark}

\subsection{Formulation by the Cahn--Hoffman vector field}\label{S2S2}

It is nontrivial to characterize the subdifferential of $TV_a$ in a given Hilbert space.
 For a general energy $\mathcal{E}$, a standard way is to propose a candidate set $A(u)$ by calculating the Euler--Lagrange operator and prove $A(u)\subset\partial\mathcal{E}(u)$.
 This part is not hard.
 The converse inclusion is difficult.
 Since $\partial\mathcal{E}$ is maximal monotone, it suffices to prove that $A$ is also maximal monotone which yields $A=\partial\mathcal{E}$.
 The proof of monotonicity is not difficult.
 To show maximality, we prove that the resolvent equation $u+\lambda Au\ni f$ is always solvable for $f\in H$ and $\lambda>0$.
 This argument is often carried out by approximation of the operator $A$.
 This procedure is found for example in \cite{GNRS}.

However, if $\mathcal{E}$ is positively homogeneous of degree one (i.e., one-homogeneous), there is an easier method due to F.\ Alter; see \cite[Chapter 1]{ACM}.
 The basic strategy is characterize the subdifferential $\partial\mathcal{E}$ by the polar $\mathcal{E}^0$ of $\mathcal{E}:H\to
(-\infty,\infty]$
\[
	\mathcal{E}^0(v) := \sup \left\{ (u,v)_H \bigm| u \in H,\ \mathcal{E}(u) \leq 1 \right\}.
\]
By convex analysis, $(\mathcal{E}^0)^0=\mathcal{E}$ if $\mathcal{E}$ is a non-negative, lower semicontinuous, convex provided that $\mathcal{E}$ is positively one-homogeneous, i.e.,
\[
	\mathcal{E}(\lambda u) = \lambda\mathcal{E}(u)
	\quad\text{for all}\quad \lambda>0,\quad u \in H.
\]
A key observation is a simple lemma.
\begin{lemma}[\cite{ACM}, Theorem 1.8] \label{LAlt}
Let $\mathcal{E}$ be convex and positively one-homogeneous in a Hilbert space $H$, then $v\in\partial_H \mathcal{E}(u)$ if and only if $\mathcal{E}^0(v)\leq1$ and $(u,v)_H=\mathcal{E}(u)$.
\end{lemma}
It is convenient to introduce a class of vector fields
\[
	X_2 = \left\{ z \in L^\infty(\Omega,\mathbb{R}^n) \bigm| \operatorname{div} z \in L^2(\Omega) \right\}.
\]
For $z\in X_2$, the normal trace $[z\cdot\nu]$ is well defined as an element of $L^\infty(\partial\Omega)$ see e.g.\ \cite{ACM}.
 In many cases, $\mathcal{E}^0$ is computable.
 For example, if $\mathcal{E}=TV$ and $H=L^2(\Omega)$,
\[
	\mathcal{E}^0(v) = \inf \left\{ \|z\|_{L^\infty} \bigm|
	z \in X_2,\ v=-\operatorname{div} z\ \text{in}\ \Omega,\ 
	[z\cdot\nu]=0\ \text{on}\ \partial\Omega \right\}
\]
at least when $\Omega=\mathbb{R}^n$, $\mathbb{T}^n$
 or a bounded domain with Lipschitz boundary \cite{ACM}.
 For higher order problem,
\[
	\mathcal{E}^0(v) = \inf \left\{ \|z\|_{L^\infty} \bigm|
	z \in L^\infty(\Omega, \mathbb{R}^n),\ v=\Delta\operatorname{div} z\ \text{in}\ \mathbb{R}^n,\ 
	\operatorname{div} z \in D_0^1 \right\},
\]
when $\mathcal{E}=TV$ and $H=D^{-1}(\mathbb{R}^n)$; see \cite[Theorem 12]{GKL}.
 Although there is no explicit literature, we expect
\[
	\mathcal{E}^0(v) = \inf \left\{ \|z\|_{L^\infty} \bigm|
	z \in X_2,\ bv=-\operatorname{div} (az)\ \text{in}\ \Omega,\ 
	a[z\cdot\nu] = 0 \ \text{on}\ \partial\Omega \right\}
\]
for general $TV_a$ and $H=L_b^2$.
 In the fractional case, we expect
\[
	\mathcal{E}^0(v) = \inf \left\{ \|z\|_{L^\infty} \bigm|
	z \in X_2,\ 
	v=-(-\Delta)^s \operatorname{div} z,\ 
	\operatorname{div} z \in \dot{H}^s(\mathbb{T}^n) \right\}
\]
for $\mathcal{E}=TV$ and $H=\dot{H}^{-s}(\mathbb{T}^n)$.
 Note that the minimizer is attained, so Lemma \ref{LAlt} implies the characterization of the subdifferential.
 We only state precise results for $\partial_{L^2}TV$ and $\partial_{D^{-1}}TV$.
\begin{thm} \label{TSub} 
\begin{enumerate}
\item[(i)] \cite[Lemma 2.4]{ACM} Assume that $\Omega=\mathbb{R}^n$, $\mathbb{T}^n$
or a bounded domain in $\mathbb{R}^n$ with Lipschitz boundary.
 An element $v\in L^2(\Omega)$ belongs to $\partial_{L^2}TV(u)$ if and only if there is $Z\in X_2(\Omega)$ such that
\item[(ia)] $|Z|\leq1$ in $\Omega$ and $[Z\cdot\nu]=0$ on $\partial\Omega$
\item[(ib)] $v=-\operatorname{div}Z$ in $\Omega$
\item[(ic)] $-(u,\operatorname{div}Z)_{L^2}=TV(u)$
\item[(i\hspace{-0.1em}i)] \cite[Theorem 14]{GKL} Assume that $\Omega=\mathbb{R}^n$.
 An element $v\in D^{-1}(\mathbb{R}^n)$ belongs to $\partial_{D^{-1}}TV(u)$ if and only if there is $Z\in L^\infty(\mathbb{R}^n)$ with $\operatorname{div}Z\in D_0^1(\mathbb{R}^n)$ such that
\item[(i\hspace{-0.1em}ia)] $|Z|\leq1$ in $\Omega$
\item[(i\hspace{-0.1em}ib)] $v=\Delta\operatorname{div}Z$ in $\Omega$
\item[(i\hspace{-0.1em}ic)] $-\langle u,\operatorname{div}Z\rangle=TV(u)$.
\end{enumerate}
Here $\langle\ ,\ \rangle$ denotes the duality pairing between $D^{-1}$ and $D_0^1$.
\end{thm}
This vector field $Z$ is often called a \emph{Cahn--Hoffman vector field}.
 If $u$ is Lipschitz continuous, the conditions (ia), (ic) (or (i\hspace{-0.1em}ia), (i\hspace{-0.1em}ib)) imply that $Z=\nabla u/|\nabla u|$ whenever $\nabla u(x)\neq0$ for almost all $x$.
 Thus, for example, for the second-order problem, the subdifferential $-\operatorname{div}Z$ formally agrees with the standard Euler--Lagrange operator $-\operatorname{div}\left(\nabla u/|\nabla u|\right)$ at least where $\nabla u(x)\neq0$.
 If $u$ is a special class of functions, the subdifferential at such $u$ is easily computable when $n=1$ including the case where there are weights $a,b$;
 see \cite{FG}, \cite{GGK}. 
 Once the subdifferential is calculated, we are able to give an explicit formulation of a solution.
\begin{thm} \label{TExp} 
\begin{enumerate}
\item[(i)] Assume that $u\in C\left([0,\infty),L^2(\Omega)\right)$.
 Then $u$ is a solution of $u_t\in-\partial_{L^2}TV(u)$ with $u_0=u(0)$ if and only if there exists $Z\in L^\infty\left(\Omega\times(0,\infty)\right)$ with
\[
	\operatorname{div}Z \in \bigcap_{\delta>0} L^2 \left( \delta,\infty; L^2(\Omega)\right)
\]
satisfying
\[
	\left\{
\begin{array}{l}
	u_t = \operatorname{div}Z \ \text{in}\ L^2(\Omega) \\	
	\left|Z(x,t)\right| \leq 1\ \text{for a.e.}\ x \in \Omega, \ 
	\left[ Z(x,t) \cdot \nu \right] = 0\ \text{for a.e.}\ x \in \partial\Omega \\
	-(u,\operatorname{div}Z)_{L^2} = TV(u)
\end{array}
	\right.
\]
for a.e.\ $t>0$.
 (If $TV(u_0)<\infty$, $\delta=0$ is allowed.)
\item[(i\hspace{-0.1em}i)] Assume that $u\in C\left([0,\infty),D^{-1}(\mathbb{R}^n)\right)$.
 Then $u$ is a solution of $u_t\in-\partial_{D^{-1}}TV(u)$ with $u_0=u(0)$ if and only if there exists $Z\in L^\infty\left(\mathbb{R}^n\times(0,\infty)\right)$ with
\[
	\operatorname{div}Z \in \bigcap_{\delta>0} L^2 \left( \delta,\infty; D_0^1(\mathbb{R}^n)\right)
\]
satisfying
\[
	\left\{
\begin{array}{l}
	u_t = -\Delta\operatorname{div}Z \ \text{in}\ D^{-1}(\mathbb{R}^n) \\	
	\left|Z(x,t)\right| \leq 1\ \text{for a.e.}\ x \in \mathbb{R}^n, \\
	\langle u,\operatorname{div}Z\rangle = -TV(u)
\end{array}
	\right.
\]
for a.e.\ $t>0$.
 (If $TV(u_0)<\infty$, $\delta=0$ is allowed.)
\end{enumerate}
\end{thm}

This follows from the characterization of the subdifferential (Theorem \ref{TSub}) except that the Cahn--Hoffman vector field should be chosen to be measurable both in $x$ and $t$.
 For the second-order case \cite[Section 2.4]{ACM}, this can be done by recalling that Bochner integrable functions can be well approximated by piecewise constant functions.
 For the fourth-order problem, although the idea is similar, the situation is slightly different as discussed in \cite[Theorem 15, Lemma 16]{GKL}.
 By the way this characterization of a solution using Cahn--Hoffman vector fields for the fourth-order of the problem was already proposed by \cite{GKM}.
 For the case $\Omega=\mathbb{T}^n$, a similar characterization of a solution is given in \cite[Theorem 1]{GMR} for general $s>0$ with $a\equiv1$.
 It is of the form.
\begin{thm} \label{TCPer}
For $s\in\mathbb{R}$, assume that $u\in C\left([0,\infty),\dot{H}_\mathrm{av}^{-s}(\mathbb{T}^n)\right)$.
 Then $u$ is a solution of \eqref{ETVP} with $a\equiv1$ and initial datum $u(0)=u_0\in \dot{H}_\mathrm{av}^{-s}(\mathbb{T}^n)$ if and only if there exists $Z\in L^\infty\left(\Omega\times(0,\infty)\right)$ with
\[
	\operatorname{div}Z \in \bigcap_{\delta>0} L^2 \left(\delta,\infty; \dot{H}_\mathrm{av}^{-s}(\mathbb{T}^n)\right)
\] 
satisfying
\[
	\left\{
\begin{array}{l}
	u_t = (-\Delta)^s \operatorname{div}Z \ \text{in}\ \dot{H}_\mathrm{av}^{-s}(\mathbb{T}^n) \\	
	\left|Z(x,t)\right| \leq 1\ \text{for a.e.}\ x \in \mathbb{T}^n \\
	\langle u,\operatorname{div}Z\rangle = -TV(u)
\end{array}
	\right.
\]
for a.e.\ $t>0$.
 (If $TV(u_0)<\infty$, $\delta=0$ is allowed).
\end{thm}
This follows from the fact that $f\in\partial_{\dot{H}_\mathrm{av}^{-s}}TV(u)$ for $u\in\dot{H}_\mathrm{av}^{-s}$ if and only if $TV^0\left((-\Delta)^{-s}f\right)\leq1$ and $\int_{\mathbb{T}^n}(-\Delta)^{-s}f\cdot v\;dx=TV(u)$, when $TV^0$ is taken in a middle space $L_\mathrm{av}^2(\mathbb{T}^n)=\dot{H}_\mathrm{av}^0(\mathbb{T}^n)$.
 For the case $s=1$, this is found in \cite[Lemma 3.1]{GK}.

As mentioned before, in the case $s=1$ and $\Omega=\mathbb{R}^n$, the gradient flow $u_t\in-\partial_{D^{-1}}TV(u)$ is not enough to study the evolution of a characteristic function for $n=1,2$.
 We have to extend the function space $D^{-1}$.
 We quickly review the way it is done in \cite{GKL}.
 For this purpose, we take $\psi\in L^2(\mathbb{R}^n)$ with compact support such that $\int_{\mathbb{R}^n}\psi\;dx\neq0$.
 We set
\[
	E_\psi^{-1} = \left\{e+c\psi \bigm| w \in D^{-1}(\mathbb{R}^n),\ c\in\mathbb{R} \right\}.
\]
This space in independent of the choice of $\psi$ so we simply write  by $E^{-1}$.
 For $n\geq3$, $E^{-1}=D^{-1}$.
 We also denote $E_0^1=D^1$ if $n\leq2$ and $E_0^1=D_0^1$ if $n\geq3$.
 The space $E^{-1}$ can be considered as a dual space of $E_0^1$.
 The inner product for $n\leq2$ is defined by
\begin{align*}
	&((v_1,v_2))_{E_0^1} := (([v_1],[v_2]))_{D_0^1}
	+ \int_{\mathbb{R}^n}\psi v_1\;dx \int_{\mathbb{R}^n} \psi v_2\; dx \\
	&((u_1,u_2))_{E^{-1}} := ((w_1,w_2))_{D^{-1}} + c_1 c_2
\end{align*}
for $u_i=w_i+c_i\psi$, $w_i\in D^{-1}$, $c_i\in\mathbb{R}$, $v_i\in E_0^1$, where $\int_{\mathbb{R}^n}\psi\;dx=1$.
 It turns out that the ``partial'' gradient flow $u_t\in-\partial_{D^{-1}}TV(u)$ is exactly what we want.
 Calculating the subdifferential $\partial_{E^{-1}}TV(u)$ and projecting to $D^{-1}$ yields $\partial_{D^{-1}}TV(u)$.
 It turns out the flow is independent of the choice of $\psi$.
 We have a characterization of this flow in terms of a Cahn--Hoffman vector field extending the one obtained in Theorem \ref{TExp}(ii). We prefer to use this characterization as a definition of a solution.
\begin{definition} \label{DEsol}
Assume that $u_0\in E^{-1}$.
 We say that $u\in C\left([0,\infty),E^{-1}\right)$ with $u(0)=u_0$ and $u_t\in L_{loc}^2\left((0,\infty),D^{-1}\right)$ is a solution of 
\begin{equation} \label{ETV4}
	u_t = -\Delta \operatorname{div} \left( \nabla u/|\nabla u| \right) \quad\text{in}\quad
	\mathbb{R}^n \times (0,\infty)
\end{equation}
with initial datum $u_0$ if there exists $Z\in L^\infty\left(\mathbb{R}^n\times(0,\infty)\right)$ with
\[
	\operatorname{div}Z \in \bigcap_{\delta>0} L^2 \left(\delta,\infty; E_0^1(\mathbb{R}^n)\right)
\]
satisfying
\[
	\left\{
\begin{array}{l}
	u_t = -\Delta\operatorname{div}Z \ \text{in}\ D^{-1}(\mathbb{R}^n) \\	
	\left|Z(x,t)\right| \leq 1\ \text{for a.e.}\ x \in \mathbb{R}^n \\
	\langle u,\operatorname{div}Z\rangle = -TV(u)
\end{array}
	\right.
\]
for a.e.\ $t>0$, where $\langle\ ,\ \rangle$ denotes the duality pairing between $E^{-1}$ and $E_0^1$. 
\end{definition}
Although there are several technical steps, we conclude the unique existence of the solution \cite[Theorem 2]{GKL}.
\begin{thm} \label{TTV4} 
For $u_0\in E^{-1}$, there exists a unique solution $u$ to \eqref{ETV4} with initial datum $u_0$.
\end{thm}
Although sometimes it is difficult to compare with original gradient flow, it is reasonable to define a solution using a Cahn--Hoffman vector field like Definition \ref{DEsol}.
 Here are a few examples of definitions.
\begin{definition} \label{DECH1}
Assume that $a:\Omega\to[0,\infty]$ is lower semicontinuous and that $b>0$ satisfies $b\in L^\infty(\Omega)$ and $1/b\in L^\infty(\Omega)$.
 We say that $u\in C\left((0,T),\mathcal{D}'(\Omega)\right)$ with $TV_a\left(u(t)\right)<\infty$ for all $t\in(0,T)$ is a Cahn--Hoffman solution (CH solution for short) of \eqref{ETV} if there exists a measurable function $Z$ on $\Omega\times(0,T)$ with $\operatorname{div}Z\in L_{loc}^2\left(\bar{\Omega}\times(0,T)\right)$ such that
\[
	\left\{
\begin{array}{l}
	bu_t = \operatorname{div}Z \\	
	\left|Z(x,t)\right| \leq a(x)\ \text{for a.e.}\ x \in \Omega,\ 
	[Z,\nu] = 0\ \text{a.e.}\ x \in \partial\Omega\ \text{for a.e.}\ t \in (0,T) \\
	-\left(u,\operatorname{div}Z\right)_{L^2} = TV_a(u)
\end{array}
	\right.
\]
for a.e.\ $t\in(0,T)$.
\end{definition}
We next consider
\begin{equation} \label{ETV4M}
	u_t = -\operatorname{div} \left(M\nabla\operatorname{div} \left(\nabla u/|\nabla u|\right)\right)
\end{equation}
where $M\in L^\infty(\Omega,\mathbb{R}^{n\times n})$ is a real symmetric matrix-valued function with $M\geq c_0 I$ uniformly with some $c_0>0$.
 The boundary condition we impose is
\begin{equation} \label{EBNN}
	\frac{\partial u}{\partial \nu} = 0, \quad
	M\cdot\nabla\operatorname{div} \left(\frac{\nabla u}{|\nabla u|}\right) = 0 \quad\text{on}\quad \partial\Omega.
\end{equation}
\begin{definition} \label{DECH2}
We say that $u\in C\left((0,T),\mathcal{D}'(\Omega)\right)$ with $TV\left(u(t)\right)<\infty$ for all $t\in(0,T)$ is a CH solution to \eqref{ETV4M} with \eqref{EBNN} if there exists $Z\in L^\infty\left(\Omega\times(0,\infty)\right)$ with  $\operatorname{div}Z \in L_{loc}^2\left(0,T,H_{loc}^1(\bar{\Omega})\right)$
\[
	\left\{
\begin{array}{l}
	u_t = -\operatorname{div}(M\nabla\operatorname{div}Z) \\	
	\left|Z(x,t)\right| \leq 1\ \text{for a.e.}\ x \in \Omega,\ 
	[Z\cdot\nu] = 0\ \text{for a.e.}\ x \in \partial\Omega \\
	\displaystyle - \langle u,\operatorname{div}Z \rangle = TV(u)
\end{array}
	\right.
\]
for a.e.\ $t>0$; here $\langle u,\operatorname{div}Z\rangle$ should be interpreted as some ``duality pair'' but for a low-dimensional problem $n\leq4$,
\[
	\langle u,\operatorname{div}Z\rangle = \int_\Omega u\operatorname{div}Z\; dx
\]
when $\Omega$ is bounded.
 In fact, in this case, $u\in L^{n/(n-1)}$ since $TV(u)<\infty$ and $\operatorname{div}Z\in L^{\frac{n}{n-2}}$ by the Sobolev embedding so $1-\frac1n+\frac12-\frac1n=\frac32-\frac2n\leq1$ for $n\leq4$.
 This implies $u\operatorname{div}Z\in L^1(\Omega)$ (for a.e.\ $t>0$).
\end{definition}

 In this definition, we do not assume that $u$ is in $L^2$.
 The problem like \eqref{ETV4M} with weight appears in relaxation process of a surface of a crystal; see e.g.\ a review paper by R.\ V.\ Kohn \cite{Kh}.

{\BLUE We mention that in the second-order case, existence and uniqueness of a solution defined in terms of a Cahn--Hoffman field can be obtained also for initial data in the non-hilbertian space $L^1(\Omega)$ larger than $L^2(\Omega)$, owing to \emph{complete $m$-accretivity} of the operator $\operatorname{div} \frac{\nabla u}{|\nabla u|}$ \cite{ABC1,ABC2}, see also \cite{ACM}. This does not seem to have an analogue in the higher-order cases. }

\section{Calibrability} \label{S3} 
{\BLUE The domain of $TV$, that is the space $BV(\Omega)$, contains functions with jump discontinuities, such as characteristic functions of sufficiently regular sets. In many cases, the evolutions of characteristic functions of sets under total variation flows are particularly simple. As a first step towards constructing such examples, we are interested in sets for which the speed $u_t(t)$ of the solution $u(t)$ is spatially constant on the set. This leads to a geometric notion of \emph{calibrability} of sets. }

In the setting of Proposition \ref{PAb}, a general theory implies that
\[
	u_t(t) = -\partial_H^0 \mathcal{E} \left(u(t)\right)
	\quad\text{for all}\quad t>0,
\]
where $\partial_H^0 \mathcal{E}$ denotes the minimal section (canonical restriction) of $\partial_H\mathcal{E}$, i.e.,
\[
	\partial_H^0 \mathcal{E}(u) := \arg\min \left\{ \|v\|_H \bigm|
	v \in \partial_H\mathcal{E}(u) \right\}.
\]
In other words, $\partial_H^0\mathcal{E}(u)$ is the
minimizer of $\|v\|_H$ over $\partial_H\mathcal{E}(u)$.
 Since $\partial_H\mathcal{E}(u)$ is a closed convex set, there always exists a unique minimizer so $\partial_H^0\mathcal{E}(u)$ is well defined.

To calculate the minimal section, we begin with a simple situation.
 Let $U$ be a smooth open set in $\Omega=\mathbb{R}^n$.
 We consider a Lipschitz function $u$ such that
\[
	\bar{U} = \left\{ x \in \Omega \bigm|
	u(x) = 0 \right\}
\]
and $u$ is smooth outside $\bar{U}$.
 Assume further that $\nabla u\neq0$ outside $\bar{U}$. 
\begin{figure}[htb]
\centering 
\includegraphics[width=8cm]{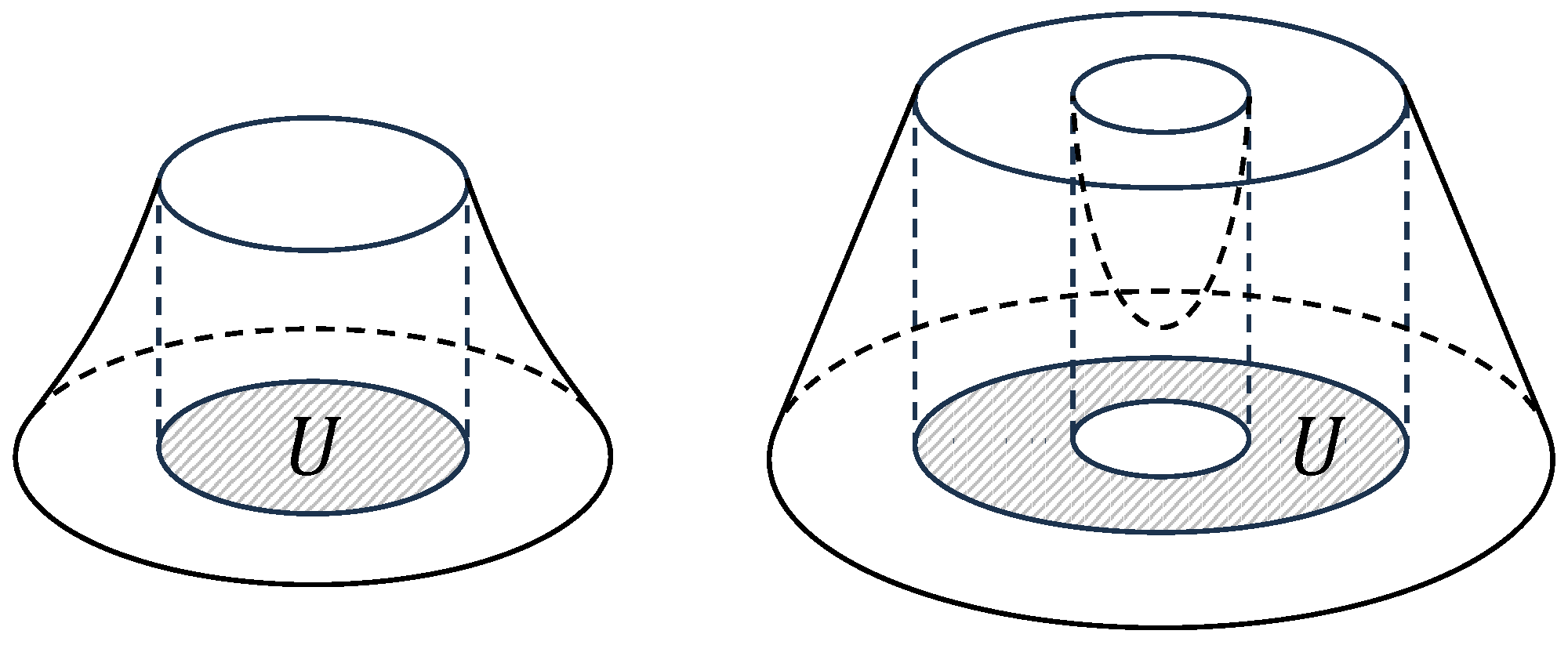}
\caption{Examples of $U$} \label{FSet}
\end{figure}

We begin with the second-order problem.
 In this case, by Theorem \ref{TSub} (i),
\[
	\partial_{L^2}^0 TV(u) = \arg\min \left\{ \|\operatorname{div}Z\|_{L^2} \bigm|
	|Z| \leq 1\ \text{a.e.\ in}\ U,\ 
	Z = \nabla u/|\nabla u|\ \text{in}\ \bar{U}^c,\ 
	Z \in X_2 \right\}. 
\]
Let $Z_0$ be a minimizer.
 Since the value of $Z_0$ outside $\bar{U}$ is always the same, it suffices to consider its restriction to $U$ (still denoted by $Z_0$).
 Then, $Z_0$ is a minimizer of
\[
	\left\{ \int_U |\operatorname{div}Z|^2\; dx \biggm|
	|Z| \leq 1\ \text{a.e.\ in}\ U,\ 
	[Z \cdot \nu] = \chi\ \text{on}\ \partial U,\ 
	Z \in X_2 \right\}. 
\]
Here, $\chi:\partial U\to\{-1,1\}$ is a signature function defined by
\[
	\chi(x) = \left\{
\begin{array}{cl}
	1 &\text{if}\ u>0\ \text{outside}\ \bar{U}\ \text{near}\ x\in\partial U \\	
	-1 &\text{otherwise}
\end{array}
	\right.
\]
For the fourth-order problem, by Theorem \ref{TSub} (i\hspace{-0.1em}i),
\[
	\partial_{D^{-1}}^0 TV(u) = \arg\min \left\{ \|\operatorname{div}Z\|_{D_0^1} \bigm|
	|Z| \leq 1\ \text{a.e.\ in}\ U,\ 
	Z = \nabla u/|\nabla u|\ \text{in}\ \bar{U}^c,\ 
	\operatorname{div}Z \in E_0^1 \right\}. 
\]
We argue in the same way.
 Let $Z_0$ be a minimizer of this problem, its restriction to $U$ (still denoted by $Z_0$) is a minimizer of
\[
	\left\{ \int_U |\nabla\operatorname{div}Z|^2\; dx \biggm|
	|Z| \leq 1\ \text{a.e.\ in}\ U,\ 
	[Z \cdot \nu] = \chi\ \text{on}\ \partial U,\ 
	\operatorname{div}Z = \chi\kappa \ \text{on}\ \partial U \right\},
\]
where $\kappa$ is the sum of all inward principal curvatures, i.e., $n-1$ times the mean curvature.
 This is because the inward and outward trace must agree for $ \operatorname{div}Z$ since $\operatorname{div}Z\in E_0^1$.
 We note that
\[
	\operatorname{div}Z = \operatorname{div} \left(\nabla u/|\nabla u|\right) = \chi \operatorname{div}\nu = \chi\kappa,
\]
where $\nu$ is an external unit normal.

We are interested in the case where $\operatorname{div}Z_0$ for the second-order problem and $\Delta\operatorname{div}Z_0$ for the fourth-order problem are constant on $U$.
 In other words, we are interested in the case that the speed of the corresponding flow on $U$ is constant.
 We rather consider a more general equation with weight for the second-order problem.
 For later convenience, we call any continuous function $\chi:\partial U\to\{-1,1\}$ a \emph{signature} of $U$ as in \cite{LMM} or \cite{GGP}.
 We first consider the second-order problem corresponding to $\partial_{L_b^2}^0 TV_a(u)$.
\begin{definition} \label{DC2ND}
Let $U$ be a smooth open set in $\mathbb{R}^n$ with signature  $\chi$.
 Assume that $a\geq0$ is continuous up to $\bar{U}$ and that $b\in L^\infty(U)$ with $1/b\in L^\infty(U)$ for $b>0$.
 We say that $U$ is $(a,b)$-\emph{calibrable} (with signature $\chi$) if there exists $Z_0$ satisfying the constraint
\[
	|Z_0| \leq 1 \quad\text{on}\quad U
\]
with boundary condition
\[
	[Z_0 \cdot \nu] = \chi \quad\text{on}\quad \partial U
\]
with the property that $\frac1b\operatorname{div}(aZ_0)$ is constant on $U$.
 We call any such $Z_0$ an $(a,b)$-calibration for $U$ (with signature $\chi$).
\end{definition}

Note that in the case $b\not\equiv1$, we still have the characterization of subdifferential $\partial_{L_b^2} TV_a(u)$ as in Theorem \ref{TSub} (i) with the modification that $v=-\left(\operatorname{div}(aZ)\right)/b$ instead of (ib) and $-\left(u,\operatorname{div}(aZ)\right)_{L^2}=TV_a(u)$ instead of (ic).
 When $b\equiv1$, the proof is given in \cite{Mol} (for general weighted anisotropic total variation), including the case when $a$ is discontinuous.
 However, as mentioned before, our definition of $TV_a$ for discontinuous $a$ is different from theirs.
 The extension to general $b$ is straightforward.
 In the one-dimensional case, this type of characterization is given in \cite{GGK} when $u$ is piecewise linear.

Our calibration $Z_0$ gives a minimizer of
\[
	\left\{ \int_U b \left|\frac{\operatorname{div}(aZ)}{b}\right|^2 dx \biggm|
	|Z| \leq 1 \ \text{a.e.\ in}\ U,\ 
	[Z \cdot \nu] = \chi\ \text{on}\ \partial U,\ 
	\operatorname{div}(aZ) \in L^2(U) \right\}.
\]
Indeed, we have by integration by parts and the Schwarz inequality
\[
	\int_{\partial U} a\chi\; d\mathcal{H}^{n-1}
	= \int_U \operatorname{div}(aZ)\; dx
	= \int_U \frac{\operatorname{div}(aZ)}{b^{1/2}} b^{1/2}\;dx
	\leq \left( \int_U \frac{\left(\operatorname{div}(aZ)\right)^2}{b}\; dx \right)^{1/2}
	\left( \int_U b\; dx \right)^{1/2}.
\]
For $Z=Z_0$, since $\operatorname{div}(aZ_0)/b$ is constant, we see $\operatorname{div}(aZ)/b^{1/2}$ and $b^{1/2}$ is parallel.
 In this case, the Schwarz inequality becomes equality.
 Thus,
\[
	\int_U b \left|\frac{\operatorname{div}(aZ)}{b} \right|^2\;dx
\]
attains its minimum value
\[
	\left( \int_{\partial U} a\chi\; d\mathcal{H}^{n-1} \right)^2 \biggm/
	\int_U b\; dx
\]
at $Z=Z_0$.

In the one-dimensional case, it is known that an interval (with signature $\chi$) is $(a,b)$-calibrable if and only if $a$ is concave with respect to a metric induced by $b$ \cite{GGK}.
 In particular, all intervals are $(a,b)$-calibrable for constant $a,b>0$.
 In higher-dimension case, the situation is more involved even for constant $a$ and $b$.
 Even if $U$ is convex, it is not necessarily $(1,1)$-calibrable.

There is a very related notion called a Cheeger set.
 We only discuss the case when $\chi\equiv1$.
 For a given open set, we set
\[
	h(U) := \inf \left\{ \frac{P(F)}{\mathcal{L}^n(F)} \biggm| F \subset U \text{ Borel,}\ 
	\mathcal{L}^n(F) \in (0,\infty) \right\},
\]
where $P(F)=TV(\mathbf{1}_F)$ is the perimeter of $F$ and $\mathcal{L}^n$ denotes the Lebesgue measure.
 Here $\mathbf{1}_F$ denotes the characteristic function of $F$, i.e., $\mathbf{1}_F(x)=1$ of $x\in F$ and $\mathbf{1}_F(x)=0$ if $x\not\in F$.
 This quantity $h(U)$ is called the \emph{Cheeger constant} and $P(F)/\mathcal{L}^n(F)$ is called the \emph{Cheeger ratio}.
 A set $F\subset U$ satisfying $P(F)/\mathcal{L}^n(F)=h(U)$ is a Cheeger set.
 If $U$ itself is a Cheeger set, $U$ is called self-Cheeger.
 It is interesting to find the value $h$ or characterize the Cheeger subset of $U$.
 Such problems are often called the Cheeger problem; see e.g.\ \cite{Leo}.

We consider the total variation flow \eqref{ETV} with $a\equiv b\equiv1$ and $\Omega=\mathbb{R}^n$ or $\mathbb{T}^n$.
 If $U$ with signature $\chi\equiv1$ is $(1,1)$-calibrable, then the speed of $u=1_U$ is equal to the Cheeger ratio {\BLUE on $U$} at least formally.
 Indeed, let $Z_0$ be a calibration.
 Then, by integration by parts, we see
\[
	P(U) = \int_{\partial U} \chi\; \mathcal{H}^{n-1} 
	= \int_U \operatorname{div}Z_0\; dx
	= \mathcal{L}^n(U) \cdot \operatorname{div}Z_0. 
\]
Thus, the speed
\[
	u_t = \operatorname{div} Z_0
\]
equals $P(U)/\mathcal{L}^n(U)$ on $U$.

It is not difficult to prove that if $U$ is $(1,1)$-calibrable, it is self-Cheeger; see e.g.\ \cite{GP}.
 The converse is also true at least in $n=2$.
 For more details, the reader is referred to \cite{ACC} or \cite{ACM} as well as \cite{GP}.
 For weighted case, the speed of $u=\mathbf{1}_U$ {\BLUE on $U$} should be the weighted Cheeger ratio
\[
	\int_{\partial U} a\chi d\mathcal{H}^{n-1} \biggm/
	\int_U b\; dx
\]
if $U$ {\BLUE with signature $\chi$ is $(a,b)$-calibrable. }

We next consider the fourth-order problem.
\begin{definition}[\cite{GKL}] \label{DC4TH}
Let $U$ be a smooth open set in $\mathbb{R}^n$ with signature $\chi$.
 We say that $U$ is ($D^{-1}$-)calibrable (with signature $\chi$) if there exists $Z_0$ satisfying the constraint
\[
	|Z_0| \leq 1 \quad\text{on}\quad U
\]
with boundary conditions
\[
	[Z_0\cdot\nu] = \chi,\quad
	\operatorname{div}Z_0 = \chi\kappa \quad\text{on}\quad \partial U,
\]
with the property that
\[
	\Delta \operatorname{div}Z_0 \quad\text{is constant over}\quad U.
\]
We call any such $Z_0$ a ($D^{-1}$-)calibration for $U$ (with signature $\chi$).
\end{definition}
If $U$ is bounded, the constant $\lambda=-\Delta\operatorname{div}Z_0$ can be characterized as a solution of the Saint-Venent problem (or the torsion problem):
\begin{equation} \label{ESVG}
	\left\{
\begin{array}{rlcl}
	-\Delta w& \hspace{-0.5em}= \lambda &\text{in} &U \\
	w & \hspace{-0.5em}= \chi\kappa &\text{on} &U
\end{array}
	\right.
\end{equation} 
by setting $w=\operatorname{div}Z_0$.
 The constant should be determined 
 by the other boundary condition:
\begin{equation} \label{ECST}
	\int_U w\; dx
	= \int_{\partial U} \chi\; d\mathcal{H}^{n-1}.
\end{equation}
In the second-order problem, the constant $\operatorname{div}(aZ_0)/b$ is determined by the Cheeger ratio.
 For the fourth-order problem, it is determined by using the solution of the Saint-Venant problem
\begin{equation} \label{ESV}
	\left\{
\begin{array}{rlcl}
	-\Delta w_\mathrm{sv}& \hspace{-0.5em}= 1 &\text{in} &U \\
	w_\mathrm{sv}& \hspace{-0.5em}= 0 &\text{on} &\partial U
\end{array}
	\right.
\end{equation}
\begin{prop}[\cite{GKL}] \label{PSp}
Let $U$ be a smooth bounded domain in $\mathbb{R}^n$.
 Assume that $Z$ is a calibration for $U$ with signature $\chi$.
 If
\[
	\lambda = -\Delta \operatorname{div} Z,
\]
then
\[
	\lambda = \left( \int_{\partial U} \chi\kappa\nu \cdot \nabla w_\mathrm{sv}\; d\mathcal{H}^{n-1}
	+ \int_{\partial U} \chi\; d\mathcal{H}^{n-1} \right) \biggm/
	\int_U w_\mathrm{sv}\; dx
\]
\end{prop}
\begin{proof}
We first note that $w_\mathrm{sv}>0$ by the maximum principle, so the denominator is not zero.
 We set $w=\operatorname{div}Z$ and decompose
\[
	w = \lambda w_\mathrm{sv} + h.
\]
Since $w$ solves \eqref{ESVG}, $h$ is a harmonic extension of $\chi\kappa$.
 The condition \eqref{ECST} gives
\[
	\chi \int_U w_\mathrm{sv}\; dx
	+ \int_U h\; dx
	= \int_{\partial U} \chi\; d\mathcal{H}^{n-1}.
\]
Thus
\[
	\lambda = \left( \int_{\partial U} \chi\; d\mathcal{H}^{n-1}
	- \int_U h\; dx \right) \biggm/
	\int_U w_\mathrm{sv}\; dx.
\]

It remains to prove
\[
	- \int_U h\; dx =	\int_{\partial U} \chi\kappa\nu\cdot\nabla w_\mathrm{sv}\; d\mathcal{H}^{u-1}.
\]
Indeed,
\begin{align*}
	- \int_U h\; dx
	= \int_U h\Delta w_\mathrm{sv}\; dx
	&= \int_{\partial U} \nu\cdot\nabla w_\mathrm{sv} \chi\operatorname{div}\nu\; d\mathcal{H}^{n-1}
	+ \int_U \nabla w_\mathrm{sv} \cdot \nabla h\; dx \\
	&= \int_{\partial U} \nu\cdot\nabla w_\mathrm{sv} \chi\kappa\; d\mathcal{H}^{n-1},
\end{align*}
since $\operatorname{div}\nu=\kappa$ and
\[
	\int_U \nabla w_\mathrm{sv}\cdot\nabla h\; dx
	= \int_{\partial U} w_\mathrm{sv} \nu\cdot\nabla h\; d\mathcal{H}^{n-1}
	- \int_U w_\mathrm{sv} \Delta h\; dx = 0-0
\]
by the definition of $w_\mathrm{sv}$ and $h$.
\end{proof}

As for the second-order problem, the calibration gives a minimizer of
\begin{equation} \label{EVar}
	\left\{ \int_U |\nabla\operatorname{div}Z|^2\;dx \biggm|
	|Z| \leq 1 \ \text{a.e.\ in}\ U,\ 
	[Z\cdot\nu] = \chi\ \text{on}\ \partial U,\ 
	\operatorname{div}Z = \chi\kappa\ \text{on}\ \partial U \right\}.
\end{equation} 
(Although a minimizer $Z^*$ is not unique, $\Delta\operatorname{div}Z^*$ is uniquely determined.)
\begin{thm}[\cite{GKL}] \label{TMin}
Let $U$ be a smooth bounded domain in $\mathbb{R}^n$.
 If $Z_0$ is a calibration of $U$ with signature $\chi$, then $Z_0$ is a minimizer of \eqref{EVar}.
\end{thm}
Compared with the second-order problem, the proof is more involved but it is still not difficult \cite[Theorem 26]{GKL}.

We conclude this section by giving examples of radial calibrable sets.

\vspace{1em}\noindent
{\bf The second-order problem.}
 We conclude the case $a\equiv b\equiv1$.
\begin{thm} \label{TCA2}
\begin{enumerate}
\item[(i)] All balls are ($(1,1)$-)calibrable.
\item[(i\hspace{-0.1em}i)] All complement of the ball are calibrable and $\lambda=0$.
\item[(i\hspace{-0.1em}i\hspace{-0.1em}i)] All annuli are calibrable.
\end{enumerate}
\end{thm}
It is not difficult to find (radial) calibration $Z(x)=z(r)x/r$, $r=|x|$.
 In the case of an open ball $B_R$ of radius $R$ (centered at the origin), we take
\[
	z(r) = \frac{\chi r}{R}
\]
so that $\operatorname{div}Z=z'(r)+\frac{n-1}{r}z(r)=\chi n/R$ and $z(R)=\chi$, $|z|\leq1$ on $(0,R)$.
 For the complement of a ball $\overline{B_R}$, it suffices to take $z(r)=\chi R^{n-1}/r^{n-1}$ so that $\operatorname{div}Z=0$.
 For an annulus $A_{R_0}^{R_1}=B_{R_1}\backslash\overline{B_{R_0}}$, we have to find $z$ such that $z'+\frac{n-1}{r}{\RED z}=r^{1-n}(r^{n-1}z)^{{\RED '}}=\lambda$ satisfying $z(R_1)=\chi_1$, $z(R_0)=-\chi_0$ and $\left|z(r)\right|\leq1$ for $r\in(R_0,R_1)$ where $\chi_i=\left.\chi\right|_{|x|=R_i}$.
 Integrating the differential equation, we observe that $z(r)=r\lambda/n-c/r^{n-1}$ with some constant $c$.
 By an explicit calculation, we are able to take $z$ satisfying the boundary condition as well as the constraint $|z|\leq1$.
 For example, consider the case $\chi_0=\chi_1=1$.
 We may assume that $R_0=1$, $R_1>1$ by scaling.
 The boundary conditions read
\[
	\frac{\lambda}{r} - c = -1, \quad
	\frac{\lambda}{n} R_1 - \frac{c}{R_1^{n-1}} = 1
\]
so that $R_1(-1+c)-cR_1^{1-n}=1$.
 In particular, $c>0$.
 We know $\lambda>0$ since $\chi\equiv1$.
 Thus, $z(r)$ is monotone increasing with $z(R_0)=-1$, $z(R_1)=1$ so it must satisfy the constraint.
 The case of indefinite signature, e.g.\ $\chi_0=-1$, $\chi_1=1$ is more involved and $c$ must have a different sign from $\lambda$.

\vspace{1em}\noindent
{\bf The fourth-order problem.}
 We consider $D^{-1}$-calibrability
\begin{thm}[\cite{GKL}] \label{TCA4}
\begin{enumerate}
\item[(i)] All balls are calibrable for $n\geq1$.
\item[(i\hspace{-0.1em}i)] All complements of balls are calibrable and $\lambda=0$ except $n=2$.
\item[(i\hspace{-0.1em}i\hspace{-0.1em}i)] If $n=2$, all complement of balls are not calibrable.
\item[(i\hspace{-0.1em}v)] All annuli (with definite signature i.e., with $\chi\equiv1$ or $\chi\equiv-1$) are calibrable except $n=2$.
\item[(v)] For $n=2$, there is $Q_*>1$ such that an annulus $A_{R_0}^{R_1}$ (with definite signature) is calibrable if and only if $R_1/R_0\leq Q_*$.
\end{enumerate}
\end{thm}
Note that by taking angular averaging, we see that a $D^{-1}$-calibration exists if and only if a radial $D^{-1}$-calibration exists for radially symmetric set \cite[Lemma 31]{GKL}.
 So to assert non-calibrability, it suffices to prove non-existence of radial calibration.
 Compared with the second-order case, we see that there occurs several exceptional phenomena for two-dimensional setting.

Instead if giving a full proof, we just give a strategy of the proof by studying the case of a ball $B_R$.
 The equation $-\Delta\operatorname{div}Z=\lambda$ is a third-order differential equation of the form
\begin{equation} \label{E3RD}
	-r^{1-n} \left(r^{n-1} \left(r^{1-n} (r^{n-1}z)' \right)' \right)' = \lambda
\end{equation}
for a radial vector field $Z(x)=z(r){\RED x}/r$ since $\operatorname{div}Z=r^{1-n}(r^{n-1}z)'$.
 For $B_R$ with $\chi\equiv-1$, the boundary conditions are
\begin{equation} \label{EBR}
	z(R) = -1, \quad x'(R) = 0
\end{equation}
since $\operatorname{div}Z=\kappa\chi$ is equivalent to saying that $z'(R)+(n-1)z(R)/R=(-1)(n-1)/R$.
 A general solution of \eqref{E3RD} is of the form
\begin{align*}
	&z(r) = c_0 r^3 + c_1 r^{3-n} + c_2 r + c_3 r^{1-n}, \quad
	c_0 = -\frac{\lambda}{{\RED 2n}(n+2)}, \quad n \neq 2, \\
	&z(r) = c_0 r^3 + c_1 r \operatorname{log}r + c_2 r + c_3 r^{-1}, \quad
	c_0 = -\lambda/16, \quad n = 2.
\end{align*}
We have to find a solution satisfying \eqref{EBR} together with the constraint $|z|\leq1$.
 The right choice is
\[
	z(r) = \frac12 \left( \frac{r}{R} \right)^3 - \frac32 \frac{r}{R}, \quad
	\lambda = -\frac{n(n+{\RED 2})}{R^3}.
\]
(The possibly singular term should be neglected so that we take $c_1=c_3=0$.
 We determine $c_2$, $\lambda$ by \eqref{EBR}.)

\section{Some simple explicit solutions} \label{S4} 

As a simple example, we consider a solution starting from $u_0=a_0\mathbf{1}_{B_{R_0}}$.

For the second-order problem \eqref{ETV} with $a\equiv b\equiv1$ or \eqref{E1TV}, the answer is very simple since $B_{R_0}$ and $\mathbb{R}^n\backslash\overline{B_{R_0}}$ are calibrable according to Theorem \ref{TCA2}.
 Let $Z_0^\mathrm{in}$, $Z_0^\mathrm{out}$ be calibrations of $B_{R_0}$ and $\mathbb{R}^n\backslash\overline{B_{R_0}}$ {\BLUE with signatures $-\operatorname{sgn}a_0$, $\operatorname{sgn}a_0$ respectively}.
 Since the normal trace of $Z_0^\mathrm{in}$, $Z_0^\mathrm{out}$ at $\partial B_{R_0}$ is continuous, there is no delta part of $\operatorname{div}Z$, where $Z=Z_0^\mathrm{in}$ in $B_{R_0}$ and $Z=Z_0^\mathrm{out}$ in $\mathbb{R}^n\backslash\overline{B_{R_0}}$.
 By definition, the speed $u_t=\operatorname{div}Z$ is constant on $B_{R_0}$ and its outside, so it is rather clear that the solution to \eqref{E1TV} starting from $u_0$ is of the form
\[
	u(x,t) = (\operatorname{sgn}a_0) \left(|a_0| - \frac{n}{R_0}t \right)_+ \mathbf{1}_{B_{{\RED R_0}}}, \quad
	c_+ = \operatorname{max}(c,0).
\]
The number $n/R_0$ is the Cheeger ratio of $B_R$ in $\mathbb{R}^n$.
 Note that the speed $\operatorname{div}Z$ outside $B_{R_0}$ equals zero.

For the fourth-order problem, \eqref{ETV4} the answer is more involved.
 Different from the second-order problem, for the minimal Cahn--Hoffman vector field $Z_0$, $\nabla\operatorname{div}Z_0$ may have jump discontinuity on $\partial B_0$ so the velocity $-\Delta\operatorname{div}Z_0$ may have a non-zero singular part concentrated on $\partial B_0$.
 In one-dimensional periodic setting, this phenomenon is already observed in \cite{GG}, \cite{Ka1}, \cite{Ka2}.
 As a consequence, even if the set $K$ of $\mathbf{1}_K$ is calibrable together with its complement, it may expand or shrink during evolution.

We shall discuss the case $n\neq2$.
 We consider $B_{R(t)}$ whose radius $R(t)$ may depend on time.
 Since $B_{R(t)}$ and its complement are calibrable by Theorem \ref{TCA4}, we take radial calibration $Z_0^\mathrm{in}$ in $B_{R(t)}$ and $Z_0^\mathrm{out}$ in $\mathbb{R}^n\backslash B_{R(t)}$.
 We set
\begin{equation*}
	Z(x,t) = \left\{
\begin{array}{ll}
	Z_0^\mathrm{in}(x), &x \in B_{R(t)} \\
	Z_0^\mathrm{out}(x), &\mathbb{R}^n \backslash B_{R(t)}.
\end{array}
	\right.
\end{equation*}
As we already observed,
\[
	Z_0^\mathrm{in}(x) = z^\mathrm{in}(r) x/r, \quad
	z^\mathrm{in}(r) = \frac12 \left(\frac{r}{R}\right)^3 - \frac32 \frac{r}{R}.
\]
Similarly,
\[
	Z_0^\mathrm{out}(x) = z^\mathrm{out}(r) x/r, \quad
	z^\mathrm{out}(r) = -\frac{n-1}{2} \left(\frac{r}{R}\right)^{{\RED 3-n}} + \frac{n-3}{2} \left(\frac{r}{R}\right)^{1-n}.
\]
This vector field $Z$ fulfills all requirements in Definition \ref{DEsol}.
 Although $\operatorname{div}Z$ is continuous across $\partial B_{R(t)}$, $\nabla\operatorname{div}Z$ may jump across $\partial B_{R(t)}$.
 Thus,
\[
	-\Delta\operatorname{div}Z = \lambda\mathbf{1}_{B_{R(t)}}
	+ \nu \cdot \left( \nabla \operatorname{div}Z_0^\mathrm{in} - \nabla \operatorname{div}Z_0^\mathrm{out} \right) \delta_{\partial B\left(R(t)\right)}
\]
with $\lambda=-n(n+2)/R^3$, where $\nu$ is the exterior unit normal of $B_{R(t)}$.
 A direct calculation shows that
\[
	\nu \cdot \left( \nabla \operatorname{div}Z_0^\mathrm{in} - \nabla \operatorname{div}Z_0^\mathrm{out} \right)
	= -\frac{n(n-4)}{R^2}.
\]
If we set
\[
	u(x,t) = a(t) \mathbf{1}_{B_{R(t)}},
\]
then $\partial_t u=\frac{da}{dt}\mathbf{1}_{B_{R(t)}} + a\frac{dR}{dt} \delta_{\partial B_{R(t)}}$.
 Since $u_t=-\Delta\operatorname{div}Z$, we end up with
\[
	\frac{da}{dt} = -\frac{n(n+2)}{R^3}, \quad
	\frac{dR}{dt} = -\frac{n(n-4)}{aR^2}.
\]
This is easy to solve.
 Indeed,
\[
	\frac{d}{dt}(aR^3) = -n(n+2) - 3n(n-4) = -n(4n-10). 
\]
Thus an explicit solution is given as
\[
	a(t) = a_0 \left( 1-\frac{n(4n-10)}{a_0 R_0^3}t \right)^{\frac{n+2}{4n-10}}, \quad
	R(t) = R_0 \left( 1-\frac{n(4n-10)}{a_0 R_0^3}t \right)^{\frac{n-4}{4n-10}}
\]
if we start with $u_0=a_0 \mathbf{1}_{B_{R_0}}$.

The case $n=2$ is more complicated.
 It was nontrivial to define a solution; see Definition \ref{DEsol}.
 Moreover, the complement of a disk is not calibrable.
 We expect that the solution becomes radially strictly decreasing for $t>0$ outside a ball $B_{R(t)}$.
 If $u$ is radially strictly decreasing outside $B_{R(t)}$, the minimal Cahn--Hoffman vector field must be $Z^\mathrm{out}(x)=-\nabla u/|\nabla u|=-x/|x|$ for $|x|>R(t)$.
 The solution to \eqref{ETV4} must satisfy
\[
	u_t = -\Delta \operatorname{div}Z^\mathrm{out}
\]
provided that $\nabla\operatorname{div}Z_\mathrm{out}\in L^2\left(\left(\overline{B_{R(t)}}\right)^c\right)$ with $\left(\overline{B_{R(t)}}\right)^c=\mathbb{R}^n\backslash\overline{B_{R(t)}}$ as in Definition \ref{DEsol}.
 We observe that
\[
	\operatorname{div}Z^\mathrm{out} = -(n-1)/|x|^2 \quad\text{and}\quad
	\nabla \operatorname{div}Z^\mathrm{out} = (n-1)x/|x|^3.
\]
If $n\leq3$, $\nabla\operatorname{div}Z^\mathrm{out}\in L^2\left(\left(\overline{B_{R(t)}}\right)^c\right)$ so it must agree with the (minimal) Cahn--Hoffman vector field for $|x|>R(t)$.
 (For $n\geq4$, $\nabla\operatorname{div}Z^\mathrm{out}$ is not in $L^2\left(\overline{B_{R(t)}}^c\right)$, so it is not a Cahn--Hoffman field.
 This indicates that if $u_0$ is radially strictly decreasing, then $u_0\not\in D(\partial_{D^{-1}}TV)$ for $n\geq4$.)
 The speed is formally equal to
\[
	u_t(x,t) = \frac{(n-1)(n-3)}{|x|^3} \quad
	x \in \mathbb{R}^n \backslash \overline{B_{R(t)}}.
\]
For $n=1$ and $3$, $u_t=0$ so the part $\overline{B_{R(t)}}^c$ cannot move.
 This is consistent that the complement of the ball is calibrable for $n=1$ and $n=3$.
 For $n=2$, the expected form of the solution of \eqref{ETV4} is
\begin{equation} \label{ESol2D}
	u(x,t) = a(t) \mathbf{1}_{B_{R(t)}} (x) + \frac{t}{|x|^3} \mathbf{1}_{\overline{B_{R(t)}}^c}(x),
\end{equation}
where
\begin{equation} \label{EODE2D}
	\frac{da}{dt} = -\frac{2\cdot4}{R^3}, \quad
	\left(a(t) - \frac{t}{R(t)^3} \right) \frac{dR}{dt}
	= \frac{2\cdot2}{R(t)^2}.
\end{equation}
As in the case for $n\geq3$, this system of ODEs provides several qualitative properties of the solution.
 Let us summarize what we observe.
\begin{thm}[\cite{GKL}] \label{TRad}
Assume that the initial datum $u_0$ is of the form
\[
	u_0 = a_0 \mathbf{1}_{B_{R_0}} \quad\text{with}\quad a_0 > 0.
\]
If $n\geq3$, then the solution $u$ to \eqref{ETV4} with initial datum $u_0$ of the form
\[
	u(x,t) = a(t) \mathbf{1}_{B_R(t)} \quad\text{for}\quad
	t < t_* = \frac{a_0 R_0^3}{n(4n-10)}
\]
and $u(x,t)\equiv0$ for $t\geq t_*$.
 Moreover, $a(t)$ is strictly decreasing and $a(t)\downarrow0$ as $t\uparrow t_*$.
 The time $t_*$ is called the extinction time.
\begin{enumerate}
\item[(i)] $R(t)$ is strictly increasing and $R(t)\uparrow\infty$ as $t\uparrow t_*$ for $n=3$.\item[(i\hspace{-0.1em}i)] $R(t)\equiv R_0$ for $n=4$.
\item[(i\hspace{-0.1em}i\hspace{-0.1em}i)] $R(t)$ is strictly decreasing and $R(t)\downarrow0$ as $t\uparrow t_*$ for $n\geq5$.
\end{enumerate}

If $n=2$, the solution is of the form \eqref{ESol2D}.
 The functions $a$ and $R$ satisfy \eqref{EODE2D}.
 In particular, there is no extinction time and $R(t)$ is strictly increasing and $a(t)$ strictly decreasing.
 Moreover, $R(t)\uparrow\infty$ and $a(t)\downarrow0$ as $t\to\infty$.
 The gap $a(t)-t/R(t)^3$ is always positive.
 See Figure \ref{F2D}.
\begin{figure}[htb]
\centering 
\includegraphics[width=7cm]{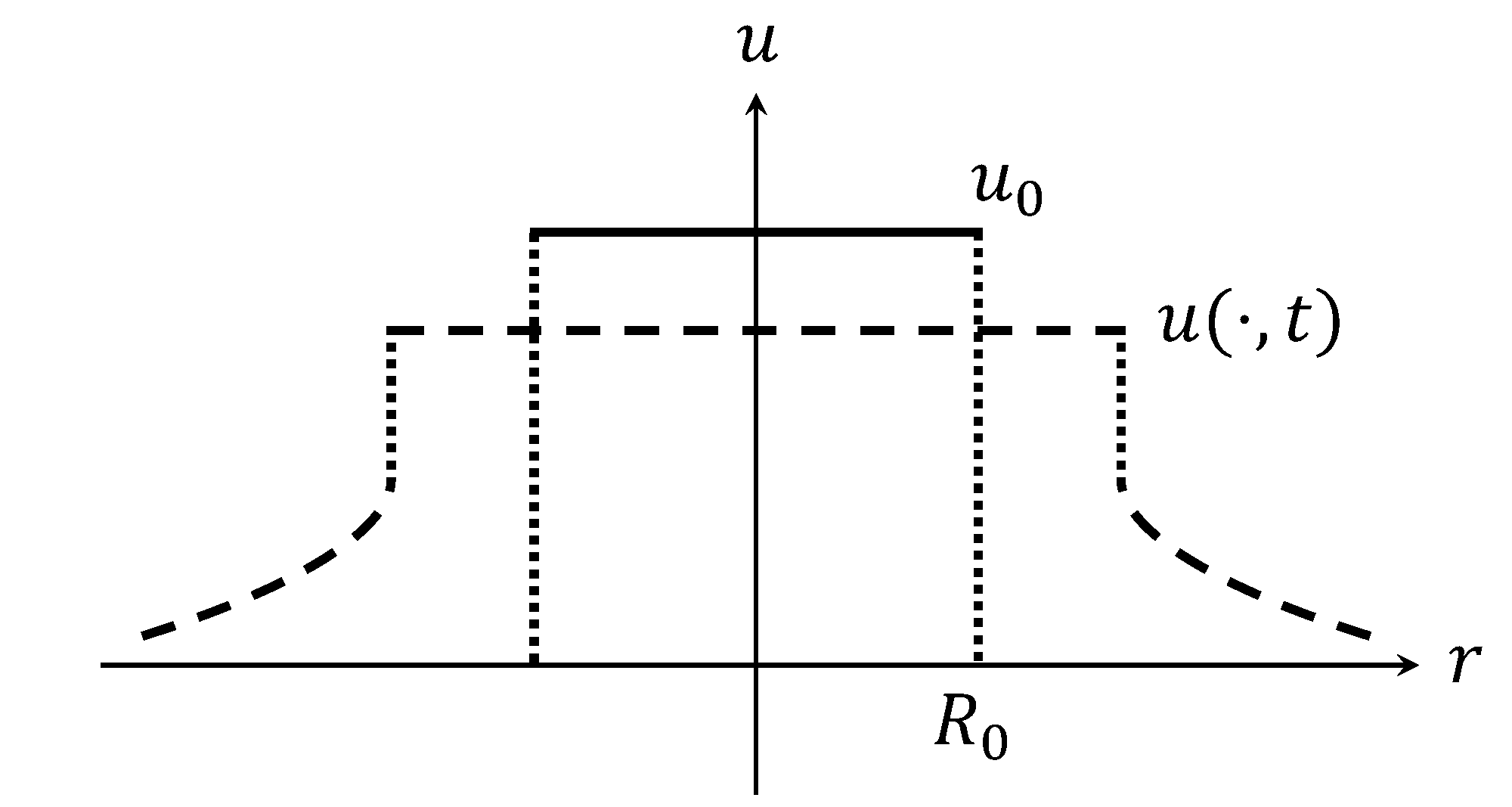}
\caption{Profile of a solution when $n=2$} \label{F2D}
\end{figure}

If $n=1$, then the solution is of the form $u(x,t)=a(t)\mathbf{1}_{B_{R(t)}}$ for $t>0$.
 There is no extinction time.
 Moreover, $R(t)$ is strictly increasing and $a(t)$ strictly decreasing with $R(t)\uparrow\infty$, $a(t)\downarrow0$ as $t\to\infty$.
\end{thm}

The reason why there is no finite extinction time for $n\leq2$ is related to the fact that $0$ is not an element of the affine space $u_0+D^{-1}$ where the flow lives if $\int_{\mathbb{R}^n}u_0\;dx\neq0$.
 We shall discuss finite extinction properties in the next section.

{\BLUE Similar analysis can be carried out for more general radially symmetric data. In the case of the second order problem, as we have checked in Theorem \ref{TCA2}, any annulus, with any choice of signature, is calibrable. Thus, by pasting together the calibrations for the annuli $A_{R_{k-1}}^{R_k}$, ball $B_{R_0}$ and ball exterior $\mathbb{R}^n\backslash B_{R_{m-1}}$, we can construct a Cahn--Hoffman vector field $Z_0$ for any piecewise constant, radially symmetric initial datum (called a \emph{stack}), i.e. 
\begin{equation} 
\label{stack_init} 
	u_0 = a_{0}^0 \mathbf{1}_{B_{R_0}} + \sum_{k=1}^{m-1} a_{0}^k \mathbf{1}_{A_{R_{k-1}}^{R_k}}
	+ a_{0}^m \mathbf{1}_{\mathbb{R}^n\backslash B_{R_{m-1}}}
\end{equation} 
with $0<R_0<R_1<\cdots<R_{m-1}$, $a_0^k\in\mathbb{R}$ ($k=0,\ldots,m$). Since the $L^2$ function $\operatorname{div} Z_0$ is constant on $A_{R_{k-1}}^{R_k}$, $B_{R_0}$ and  $\mathbb{R}^n\backslash B_{R_{m-1}}$, this shows that 
for small $t> 0$ 
\begin{equation} \label{stack_evo}
u(t) = a^0(t) \mathbf{1}_{B_{R_0}} + \sum_{k=1}^{m-1} a^k(t) \mathbf{1}_{A_{R_{k-1}}^{R_k}}
	+ a^m(t) \mathbf{1}_{\mathbb{R}^n\backslash B_{R_{m-1}}},
 \end{equation} 
with $d a^k/dt$ constant, depending only on $R_{k-1}$, $R_k$ and the signs of $a_{0}^{k+1} - a_{0}^k$ and $a_{0}^k - a_{0}^{k-1}$ for $k = 1, \ldots, m-1$ ($d a^0/dt$ and $d a^m/dt$ are also constant, depending only on $R_0$, $\operatorname{sgn}(a^1 - a^0)$, and $R_m$, $\operatorname{sgn}(a^m - a^{m-1})$, respectively. This determines the evolution until the first time instance $t>0$ such that $a^k(t) = a^{k+1}(t)$ for some $k \in \{0, \ldots, m-1\}$ (first \emph{merging time}). Then, the solution is again of form \eqref{stack_evo}, with a smaller $m$, and one can repeat the procedure until the solution becomes constant. 

In \cite{GKL}, the evolution of stacks under the fourth-order TV flow was studied in detail. In particular, it has been proved that in dimensions $n \neq 2$, if the initial datum is of form \eqref{stack_init}, i.e.~if $u_0$ is a stack, then $u(t)$ is also a stack for $t>0$. The same does not hold in $n=2$ as evidenced already by evolution of characteristic functions of balls. 

In the second-order case, one can produce more complicated examples of explicit piecewise constant solutions with initial data such as characteristic functions of sums of calibrable sets which are distant enough from each other \cite{BCN}. In the fourth-order case, we do not know about any non-radial examples of explicit solutions in $n \geq 2$.  

In the case of bounded domains, the solutions can be more complicated. However, in the 1D setting (where, to be fair, the only connected bounded domains are intervals) one can produce explicit solutions for a dense set of initial data: the step functions. Let us present the construction in the case of periodic boundary condition, i.e.~$\Omega= \mathbb{R}/\mathbb{Z}$, for the second-order flow. Take $u_0$ a step function, that is 
\begin{equation} \label{step_init} 
	u_0 = a_{0}^1 \mathbf{1}_{[x_0,x_1)} + a_{0}^2 \mathbf{1}_{[x_1,x_2)} + \ldots +a_{0}^m \mathbf{1}_{[x_{m-1},x_m)},
\end{equation}
where $x_m = x_0 + 1$. 
We can assume that $a_{0}^{k-1} \neq a_{0}^k$ for $k=1,\ldots,m$ and $a_{0}^m \neq a_{0}^1$.   Since the intervals are calibrable, for small $t>0$ (until the first merging time) we have  
\begin{equation} 
\label{step_evo} 
	u(t) = a^1(t) \mathbf{1}_{[x_0,x_1)} + a^2(t) \mathbf{1}_{[x_1,x_2)} + \ldots +a^m(t) \mathbf{1}_{[x_{m-1},x_m)}
\end{equation} 
with $a^k$ evolving at constant speed
\[ d a^k /dt = \theta_k /(x_k - x_{k-1}) \text{ for } k=1,\ldots,m,\]
\begin{equation}\label{theta_k} 
\theta_k = \left\{ \begin{array}{ll} 
+2 &\text{if } a^k_0 < a^{k-1}_0 \text{ and }  a^k_0 < a^{k+1}_0, \\ 
-2 &\text{if } a^k_0 > a^{k-1}_0 \text{ and }  a^k_0 > a^{k+1}_0, \\
0 & \text{otherwise.}
\end{array}\right. 
\end{equation} 
Then, as in the radially symmetric case, we can show that the solution remains a step function throughout the evolution. 
}
\section{Upper bounds for the extinction time} \label{S5} 

In many examples, the solution may have a finite extinction time.
 We consider this problem both for the second-order and the fourth-order problem.
 For an initial datum $u_0$, the \emph{extinction time} of a solution $u$ is defined as
\[
	T^*(u_0) = \inf \left\{ t \in (0,\infty) \bigm| u(x,\tau) = 0
	\ \text{for}\ \tau \geq t \right\}.
\]

\subsection{Second and fourth-order problems} \label{S5S1} 

We consider the second-order problem \eqref{E1TV}.
 We first give an upper bound for the extinction time given by \cite[Theorem 2.4, Theorem 2.5]{GK}.
 Let $S_n$ be the best constant of the Sobolev (isoperimetric) inequality
\[
	\| v \|_{n/(n-1)} \leq S_n TV(v),
\]
where $\|v\|_p=\left(\int_{\mathbb{R}^n}|v|^p\;dx\right)^{1/p}$.
\begin{thm}[Second-order problem] \label{TUB2}
\begin{enumerate}
\item[(i)] Assume that $n\geq2$ and $u_0\in L^2(\mathbb{R}^n)$.
 Then
\[
	T^*(u_0) \leq S_n \|u_0\|_n.
\]
\item[(i\hspace{-0.1em}i\hspace{-0.1em}i)] Assume that $n=1$.
 Then $\|u\|_1(t)\leq\|u_0\|_1-t$, where $u$ is the solution to \eqref{E1TV}.
 In particular,
\[
	T^*(u_0) \leq \|u_0\|_1.
\]
If $u_0\in L^2(\mathbb{R})$ does not belong to $L^1(\mathbb{R})$, the solution $u$ may not have finite extinction time, i.e., $T^*(u_0)=\infty$.
\end{enumerate}
\end{thm}
Idea of the proof.
 In the case $n=2$.
 By the abstract definition of a solution, we know that
\[
	\frac12 \frac{d}{dt} \int_{\mathbb{R}^2} |u|^2\;dx
	= (u, u_t)_{L^2} = -\left( u, \partial_{L^2}TV(u) \right)_{L^2}.
\]
Since $TV(u)$ is positively one-homogeneous, we see $\left(u, \partial_{L^2}TV(u)\right)=TV(u)$
 (see Lemma \ref{LAlt}).
 In particular,
\[
	\frac12 \frac{d}{dt} \int_{\mathbb{R}^2} |u|^2\;dx
	= -TV(u).
\]
We apply the Sobolev inequality to get
\[
	\frac12 \frac{d}{dt} \int_{\mathbb{R}^2} |u|^2\;dx
	\leq -\frac{1}{S_2} \left( \int_{\mathbb{R}^2} |u|^2\;dx \right)^{1/2},
\]
which yields
\[
	\frac{d}{dt} \| u \|_2(t) \leq -S_2^{-1}
	\quad\text{provided that}\quad \| u \|_{L^2}(t) \neq 0.
\]
This yields
\[
	\| u \|_2(t) \leq \| u_0 \|_2 - S_2^{-1} t
\]
which implies the desired estimate for $T^*(u_0)$.

In the case $n\geq3$, we formally multiply $|u|^{n-2}u$ to \eqref{E1TV} to get
\[
	\frac1n \frac{d}{dt} \int_{\mathbb{R}^n} |u|^n\;dx
	= \int_{\mathbb{R}^n} |u|^{n-2}uu_t\;dx
	= \int_{\mathbb{R}^n} |u|^{n-2}u \operatorname{div} \left(\nabla u/|\nabla u|\right)dx.
\]
Integrating by parts, the right-hand side becomes
\[
	-(n-1) \int_{\mathbb{R}^n} |u|^{n-2} \nabla u\cdot\nabla u/|\nabla u|
	= -(n-1) \int_{\mathbb{R}^n} |u|^{n-2} |\nabla u|\;dx
	= - \int_{\mathbb{R}^n} \left| \nabla|u|^{n-2} u \right| dx.
\]
Then applying the Sobolev inequality for $v=|u|^{n-2}u$ to get
\[
	\frac1n \frac{d}{dt} \int_{\mathbb{R}^n} |u|^n\;dx
	\leq -S_n^{-1} \left( \int_{\mathbb{R}^n} |u|^n\;dx \right)^{(n-1)/n},
\]
i.e.,
\[
	\frac1n \frac{d}{dt} \| u \|_n^n
	\leq -S_n^{-1} \| u \|_n^{n-1}.
\]
Thus, as for $n=2$, we have $\|u\|_n(t)\leq\|u_0\|_n-S_n^{-1}t$ if $\|u\|_n(t)\neq0$, and the desired estimate holds.
 The argument for $n\geq3$ is formal because we do not know whether multiplication by $|u|^{n-2}u$ is justified since $u$ may not be an $L^n$-valued absolutely continuous function of $t$.
 Fortunately, our argument is justified by approximation of the equation by smooth uniformly parabolic equations as in \cite{GK}.
 In one-dimensional setting, we approximate $\int|u|\;dx$ by $\int f(u)$ where $f$ is a convex function; see \cite[Section 2.5]{GK}. 
\begin{remark}[\cite{GK}] \label{RPN}
The results still hold when $\Omega=\mathbb{T}^n$ or a smooth bounded domain with the Neumann boundary condition for average-free $L^2$ initial data with possibly different value of the best Sobolev constant $S_n$.
 Note that $S_n$ may depend on the shape of $\Omega$ but it is scale invariant in the sense that $S_n$ is invariant under dilation, i.e., $S_n(\lambda\Omega)=S_n(\Omega)$ for any $\lambda>0$, where $\lambda\Omega=\left\{ \lambda x \mid x\in\Omega \right\}$.
 For the Dirichlet problem, it still holds for $n=2$ but there is no literature claiming the estimate $T^*(u_0)\leq S_n\|u_0\|_n$ for $n\geq3$; 
justification of the formal estimate will be difficult since boundary detachment phenomenon is expected unless the domain is mean-convex.

For the gradient flow of $p$-Dirichlet energy for $p>1$, i.e., $u_t\in-\partial E_p(u)$, an extinction time estimate
\[
	T^*(u_0) \leq C \|u_0\|_s^{2-p}, \quad
	s = n(2-p)/p, \quad
	n \geq 2, \quad
	1<p\leq \frac{2n}{n+1}
\]
has been proved by a similar method \cite[Proposition 2.1, Proposition 3.1 of Chapter V\hspace{-0.1em}I\hspace{-0.1em}I]{DiB} both for the Cauchy problem $\Omega=\mathbb{R}^n$ and the Dirichlet problems for a bounded domain.
\end{remark}
\begin{remark} \label{RCom}
Since \eqref{E1TV} has a comparison principle, the estimate $T^*(u_0)<\infty$ is often proved by comparison with the evolution of the characteristic function $a(t)\mathbf{1}_{B_{R_0}}$.
 For example, suppose that $u_0\geq0$ and $u_0\leq a_0$ and $u_0\equiv0$ outside $B_{R_0}$.
 Then
\[
	0 \leq u(x,t) \leq a(t) \quad\text{in}\quad B_{R_0}
\]
and $a(t)=\left(a_0-(n/R_0)t\right)_+$.
 Thus $T^*(u_0)\leq(\sup u_0)R_0/n$.
\end{remark}

We next study the fourth-order problem.
 We first calculate the growth of $L^p$-norm in a formal way
\begin{align*}
	\frac{1}{p(p-1)} \frac{d}{dt} \int_{\mathbb{R}^n} |u|^p\;dx
	&= \frac{1}{p-1} \int_{\mathbb{R}^n} |u|^{p-2} u u_t\;dx \\
	&= \frac{1}{p-1} \int_{\mathbb{R}^n} |u|^{p-2} u (-\Delta\operatorname{div}z)\;dx
	\quad\text{with}\quad z = \nabla u/|\nabla u|
\end{align*}
since \eqref{ETV4} is of the form $u_t=-\Delta\operatorname{div}z$.
 Integrating by parts, the right-hand side equals
\[
	\int_{\mathbb{R}^n} |u|^{p-2} \nabla u\cdot\nabla\operatorname{div}z\;dx
	= -\int_{\mathbb{R}^n} |u|^{p-2} \nabla^2 u: \nabla\otimes z\;dx
	-(p-2) \int_{\mathbb{R}^n} |u|^{p-4} u \nabla\otimes z: \nabla u\otimes\nabla u\;dx,
\]
where we assume that effect at space infinity does not appear; 
here $A:B=\operatorname{trace}(AB^T)$ for matrices $A$ and $B$ and $\nabla\otimes z$ for a vector field $z=(z_1,\ldots,z_n)$ denotes a matrix $(\partial_i z_j)$ for $\partial_i=\partial/\partial x_i$.
 Since $z$ is a subgradient of a positively one-homogeneous function of $\nabla u$ we see that $(\nabla\otimes z)(\nabla u)^T=0$ for $\nabla u=(\partial_1u,\ldots,\partial_nu)$ from the Euler's identity for a positively zero-homogeneous function.
 Here is a more explicit argument.
 Since
\[
	\nabla\otimes z = \nabla\otimes \frac{\nabla u}{|\nabla u|} = \frac{1}{|\nabla u|}  \nabla^2 u \left( I - \frac{\nabla u}{|\nabla u|} \otimes \frac{\nabla u}{|\nabla u|} \right) 
\]
it is rather clear that $(\nabla\otimes z)(\nabla u)^T=0$ since $P=I-\nabla u\otimes\nabla u/|\nabla u|^2$ is a projection orthogonal to $\nabla u$.
 We end up with
\[
	\frac{1}{p(p-1)} \frac{d}{dt} \int_{\mathbb{R}^n} |u|^p\;dx
	= -\int_{\mathbb{R}^n} |u|^{p-2} \nabla^2 u: \nabla\otimes z\;dx \leq 0  
\]
since $|\nabla u|\nabla\otimes z=\nabla^2uP$ and $P$ is a non-negative symmetric matrix.
 This is a formal argument that needs to be rigorously justified, see \cite{GKL2}. Eventually, we have
\begin{lemma} \label{LLPB}
Assume that $u_0\in L^p(\mathbb{R}^n)\cap E^{-1}$.
 Let $u$ be the solution of \eqref{ETV4} with initial datum $u_0$.
 Then $\|u\|_p(t)$ is non-increasing in $t$ for $1\leq p<\infty$.
 In particular, $\|u\|_p(t)\leq\|u_0\|_p$ for all $t\geq0$.
\end{lemma}

We now discuss an upper bound for the extinction time.
 Since the extinction time for $n\leq2$ may be infinite as observed in Section \ref{S4}, we assume $n\geq3$.
 In this case, our fundamental identity
\[
	\frac12 \frac{d}{dt} \|u\|_{D^{-1}}^2 (t) = -TV(u) 
\]
is obtained by taking inner product of $u$ with \eqref{ETV4}.
 By the Sobolev inequality, we see that
\begin{equation} \label{EKI}
	\frac12 \frac{d}{dt} \|u\|_{D^{-1}}^2 (t) 
	\leq -S_n^{-1} \|u\|_{n/(n-1)} 
\end{equation}
as for the second-order problem.
Again by the Sobolev inequality
\[
	\|u\|_{2^*} \leq C_n \|u\|_{D_0^1}
\]
for $2^*=2n/(n-2)$, hence
\begin{align} 
	\|u\|_{D^{-1}} &= \sup \left\{ \langle u,v \rangle \Bigm|
	\|v\|_{D_0^1} \leq 1 \right\} \notag\\
	&\leq C_n \sup \left\{ \langle u,v \rangle \bigm|
	\|v\|_{2^*} \leq 1 \right\}
	= C_n \|u\|_{(2^*)'} \label{EDSob}
\end{align}
where $(2^*)'=2n/(n+2)$.
 In the case $n=4$, $(2^*)'=4/3=n/(n-1)$ so \eqref{EKI} yields
\[
	\frac12 \frac{d}{dt} \|u\|_{D^{-1}}^2(t)
	\leq -S_4^{-1} C_4^{-1} \|u\|_{D^{-1}}(t)
\]
and this implies
\[
	T^*(u_0) \leq S_4 C_4 \|u\|_{D_0^1}.
\]
This type of estimate is already obtained in \cite{GK} when $\Omega=\mathbb{T}^n$.
 For other $n$, we recall the H\"older inequality
\[
	\|u\|_{(2^*)'}	\leq \|u\|_{n/(n-1)}^\theta \|u\|_p^{1-\theta}
\]
with $\frac{n+2}{2n}=\frac{n-1}{n} \theta + \frac1p (1-\theta)$ for $n\geq3$.
 If $n=4$, $(2^*)'=4/3= n/(n-1)$ so $\theta=1$.
 If $n\geq 5$ we have to take $p > (2^*)'$, and if $n=3$ we have to take $p<6/5=(2^*)'$ so that $0<\theta<1$.
 By \eqref{EKI} and \eqref{EDSob}, we now obtain 
\begin{align*}
	\frac12 \frac{d}{dt} \|u\|_{D^{-1}}^2(t)
	&\leq -S_n^{-1} \left( C_n^{-1} \|u\|_{D^{-1}} \right)^{1/\theta} \Bigm/\|u\|_p^{(1-\theta)/\theta} \\
	&\leq -S_n^{-1} \left( C_n^{-1} \|u\|_{D^{-1}} \right)^{1/\theta} \Bigm/ \|u_0\|_p^{(1-\theta)/\theta}.
\end{align*}
In the last inequality, we invoke Lemma \ref{LLPB}.
 In other words,
\[
	\frac{d}{dt} \|u(t)\|_{D^{-1}} 
	\leq -A_\theta^{-1} \|u(t)\|_{D^{-1}}^{1/\theta-1} \bigm/ \|u_0\|_p^{1/\theta-1}, \quad
	A_\theta = S_n C_n^{1/\theta}.
\]

The differential inequality
\[
	\frac{dy}{dt} \leq -ky^\sigma, \quad
	y(0) = y_0 > 0,
\]
with $\sigma\in[0,1)$, $k>0$ yields the estimate
\[
	y(t) \leq \left( y_0^{1-\sigma} - (1-\sigma) kt \right)_+^\frac{1}{1-\sigma}.
\]
In particular, $y(t)$ must be zero for $t\geq t_*=y_0^{1-\sigma}\bigm/(1-\sigma)k$.
 Applying this estimate to our differential inequality for $y(t)=\|u(t)\|_{D^{-1}}$ implies that
\[
	\|u(t)\|_{D^{-1}} \leq \left( \|u_0\|_{D_0}^{2-1/\theta} - (2-1/\theta) A_\theta^{-1} t \bigm/ \|u_0\|_p^{(1-\theta)/\theta} \right)_+^{\theta/(2\theta-1)}
\]
provided that $0\leq\frac{1}{\theta}-1<1$, which is equivalent to
\begin{equation} \label{EEIN}
	\frac12 < \theta = \left(\frac{n+2}{2n} - \frac1p \right) \biggm/
	\left(\frac{n-1}{n} - \frac1p \right) \leq 1.
\end{equation}
In this case, we have an upper bound for the extinction time
\begin{equation} \label{EEXT}
	T^*(u_0) \leq \frac{A_\theta\theta}{2\theta-1} \|u_0\|_p^{1/\theta-1}
	\|u_0\|_{D^{-1}}^{2-1/\theta}.
\end{equation}
It remains to check the validity of the inequality \eqref{EEIN}.
 If $n=4$ so that $(n+2)/2n=(n-1)/n$, $\theta$ must be $1$ and \eqref{EEXT} is reduced to what we already obtained.
 Since $\theta>1/2$ in \eqref{EEIN} can be written as
\[
	\frac{n+2}{2n} - \frac{n-1}{2n} > \frac{1}{2p}, \quad
	\text{i.e.,}\quad \frac1p < \frac3n,
\]
the estimate \eqref{EEXT} holds for all $p>n/3$ provided that $n\geq5$.
 In the case $n=3$, \eqref{EEIN} yields that $1<p<6/5=(2^*)'$.
 Summarizing what we discussed, we obtain at least formally
\begin{thm} \label{TUB4}
Let $p>n/3$ for $n=5$ and $p\in(1,6/5)$ for $n=3$.
 Assume that
\[
	\theta = \left(\frac{n+2}{2n} - \frac1p \right) \biggm/
	\left(\frac{n-1}{n} - \frac1p \right)
\]
for $n\geq3$ and $n\neq4$.
 Then
\[
	T^*(u_0) \leq \frac{A_\theta\theta}{2\theta-1} \|u_0\|_p^{1/\theta-1} \|u_0\|_{D^
{-1}}^{2-1/\theta}
\]
with $A_\theta=S_n C_n^{1/\theta}$ provided that $u_0\in L^p(\mathbb{R}^n)\cap D^{-1}(\mathbb{R}^n)$.
 In the case $n=4$,
\[
	T^*(u_0) \leq A_1 \|u_0\|_{D^{-1}}.
\]
\end{thm}
This result is consistent with our explicit solutions in Section \ref{S4}.
 Our examples in Section \ref{S4} for $n\leq2$ shows that Theorem \ref{TUB4} cannot be extended to $n\leq2$.
 Note that our estimate is scale-invariant.

The results easily extend to the case $\Omega=\mathbb{T}^n$ by considering average free spaces.
 For a smooth bounded domain $\Omega$, there are several possible boundary conditions
\begin{enumerate}
\item[(DD)] Dirichlet--Dirichlet: $u=0$, $\operatorname{div}\left(\frac{\nabla u}{|\nabla u|}\right)=0$ on $\partial\Omega$;
\item[(ND)] Neumann--Dirichlet: $\frac{\partial u}{\partial\nu}=0$, $\operatorname{div}\left(\frac{\nabla u}{|\nabla u|}\right)=0$ on $\partial\Omega$;
\item[(DN)] Dirichlet--Neumann: $u=0$, $\frac{\partial}{\partial\nu} \operatorname{div}\left(\frac{\nabla u}{|\nabla u|}\right)=0$ on $\partial\Omega$;
\item[(NN)] Neumann--Neumann: $\frac{\partial u}{\partial\nu} =0$, $\frac{\partial}{\partial\nu} \operatorname{div}\left(\frac{\nabla u}{|\nabla u|}\right)=0$ on $\partial\Omega$.
\end{enumerate}
{\BLUE
In (DD) and (DN) case, as well as in the (NN) case if $\Omega$ is convex, the boundary terms in the calculation of $\frac{d}{dt}\int_\Omega|u|^p\;dx$ vanish, so we still expect monotonicity of $\|u\|_{L^p(\Omega)}$.
 Note that the formulation of flow for (DN) and (NN) itself is non-trivial.
 We shall discuss the formulation in the forthcoming paper \cite{GKL2}.
 In \cite{GK}, an upper bound for the extinction time in a periodic domain is obtained using a Sobolev space of negative order, and it is extended in \cite{GKM} to a bounded domain under (DD).
 In these settings, even in $n\leq2$ the solution has a finite extinction time.
 This is a big difference between $\mathbb{R}^n$ case and a bounded domain (or $\mathbb{T}^n$).
}

\subsection{Fractional case} \label{S5S2} 

We next consider the fractional case \eqref{ETVP} with constant weight, i.e., $a\equiv1$.
 In this case, we mimic the way to derive an upper bound for the extinction time by using negative Sobolev norm in the case $s=1$ discussed in \cite{GK} and \cite{GKL}, since it is not clear whether $L^p$-norms of the solution are well controlled or not.

We begin with a fundamental identity.
\begin{prop} \label{PFF} 
Let $u$ be a solution to \eqref{ETVP} (with $a\equiv1$) with initial datum $u_0\in\dot{H}_\mathrm{av}^{-s}(\mathbb{T}^n)$.
 Then
\[
	\frac12 \frac{d}{dt} \|u(t)\|_{\dot{H}_\mathrm{av}^{-s}(\mathbb{T}^n)}^2
	= -TV(u(t)) \quad\text{for a.e.}\quad t>0.
\]
\end{prop}

This is easily obtained from $u_t\in-\partial_{\dot{H}_\mathrm{av}^{-s}}TV(u)$ by taking $\dot{H}^{-s}$ inner product with $u$.

The basic strategy is the same as in Section \ref{S5S1}.
 We shall estimate the right-hand side by an interpolation inequality:
\[
	-TV(u) \leq -\|u\|_{\dot{H}_\mathrm{av}^{-s}}^{1+\alpha} \Bigm/ \|u\|_X^\alpha
\]
with a suitable norm $\|u\|_X$, which does not grow quickly as $t$ increases.
 In Section \ref{S5S1}, we take $\|u\|_X$ just $L^p$-norm since $\|u(t)\|_p$ is not increasing.
 In this section, we instead take a negative Sobolev-norm.

We begin with a simple setting when $n=2(s+1)$, where an interpolation inequality is unnecessary.
 This corresponds to the case $n=4$ for $s=1$.
 We recall a fractional Sobolev inequality
\[
	\|u\|_{L^q} \leq C_{n,s} \|u\|_{\dot{H}_\mathrm{av}^s}
	\quad\text{for}\quad u \in \dot{H}_\mathrm{av}^s(\mathbb{T}^n),
\]
where $\frac1q=\frac12-\frac{s}{n}$ for $s>0$, $q<\infty$.
 As discussed when deriving \eqref{EDSob} for $D^{-1}$, by duality we observe that
\[
	\|u\|_{\dot{H}_\mathrm{av}^{-s}} \leq C_{n,s} \|u\|_{L^{q'}},\quad
	1/q + 1/q'=1.
\]
If $s>0$ satisfies $n=2(s+1)$, then $q=n$ and $q'=n/(n-1)$.
 Since $\|u\|_{L^{\BLUE n/(n-1)}
 }\leq S_nTV(u)$ by the Sobolev (isoperimetric) inequality, Proposition \ref{PFF} implies that
\[
	\frac12 \frac{d}{dt} \|u\|_{\dot{H}_\mathrm{av}^{-s}(\mathbb{T}^n)}^2 (t)
	\leq -C_{n,s}^{-1} S_n^{-1} \|u\|_{\dot{H}_\mathrm{av}^{-s}(\mathbb{T}^n)}.
\]
We thus obtain an upper bound for the extinction time.
\begin{prop} \label{PExt}
Let $u$ be a solution to \eqref{ETVP} (with $a\equiv1$) with initial datum $u_0\in\dot{H}_\mathrm{av}^{-s}(\mathbb{T}^n)$.
 Assume that $n=2(s+1)$.
 Then,
\[
	T^*(u_0) \leq C_{n,s} S_n \|u_0\|_{\dot{H}_\mathrm{av}^{-s}(\mathbb{T}^n)}.
\]
\end{prop}

We next derive an interpolation inequality, which is an extension of the inequality obtained in \cite{GK} for $s=1$.
 We define a homogeneous negative Sobolev norm as
\[
	\|w\|_{\dot{W}_\mathrm{av}^{-1,p}}
	= \sup \left\{ \langle w,\varphi \rangle \Bigm|
	\varphi \in C_\mathrm{av}^1(\mathbb{T}^n),\
	\|\nabla\varphi\|_{L^{p'}} \leq 1\right\}
\]
when $w$ is a distribution on $\mathbb{T}^n$.
\begin{lemma} \label{LInt}
Assume that $s$, $p$, $\theta$ satisfy
\[
	1 \leq n \leq 2(s+1), \quad
	0 < s \leq 1, \quad
	1 \leq p \leq \infty, \quad
	\frac12 \leq \theta \leq 1
\]
and the scaling balance
\[
	s + \frac{n}{2} = (1-\theta) \left(2s+1+\frac{n}{p} \right) + \theta(n-1).
\]
Then there is a constant $C_*$ such that
\[
	\|u\|_{\dot{H}_\mathrm{av}^{-s}}
	\leq C_* \left\| (-\Delta)^{-s}u \right\|_{\dot{W}_\mathrm{av}^{-1,p}}^{1-\theta} TV(u)^\theta
	\quad\text{for all}\quad u \in \dot{H}_\mathrm{av}^{-s}(\mathbb{T}^n) \cap BV(\mathbb{T}^n).
\]
The constant $C_*$ is {\BLUE invariant under dilation }
in the sense that it is independent of $\lambda>0$ if $u$ is replaced by $u_\lambda\left(=u(\lambda x)\right)$.
\end{lemma}

{\BLUE The scaling balance is a consequence of invariance of $C_*$.
 It can be rewritten as}
\[
	\theta = \frac12 + \left\{ 2 + 2p \left( \frac{2(s+1)}{n}-1 \right) \right\}^{-1}.
\]
Thus the assumption $\frac12\leq\theta\leq1$ is redundant since it follows from $n\leq2(s+1)$, $1\leq1'\leq\infty$ and the scaling balance.
 We note that
\begin{align*}
	&\theta=1 \Longleftrightarrow n=2(s+1) \\
	&\theta=\frac12 \Longleftrightarrow 1 \leq n < 2(s+1),\quad
	p=\infty,
\end{align*}
since $s\leq1$.
 The idea of the proof is parallel to that of \cite{GK}.
 We give the proof in the case $n < 2(s+1)$ for the reader's convenience.
\begin{proof}
We decompose $u$ into regular part $u_\mathrm{reg}$ and singular part $u_\mathrm{sing}$
\[
	u = u_\mathrm{reg}-u_\mathrm{sing}, \quad
	u_\mathrm{reg} := e^{t\Delta} u, \quad
	u_\mathrm{sing} := \int_0^t \Delta e^{\tau\Delta}u\; d\tau
\]
based on the formula
\[
	e^{t\Delta}u - u = \int_0^t \frac{d}{d\tau} e^{\tau\Delta}u\; d\tau
	= \int_0^t \Delta e^{\tau\Delta}u\; d\tau.
\]
(This idea is standard to prove an interpolation inequality as in \cite[Chapter 6]{GGS}.)
 Since
\[
	\|u\|_{\dot{H}_\mathrm{av}^{-s}}^2 
	= \int_{\mathbb{T}^n} (-\Delta)^{-s} u\cdot u\; dx
	= \int_{\mathbb{T}^n} (-\Delta)^{-s} u\cdot u_\mathrm{reg}\; dx
	- \int_{\mathbb{T}^n} (-\Delta)^{-s} u\cdot u_\mathrm{sing}\; dx,
\]
we estimate each term separately.

For the regular part, by definition of $\dot{W}_\mathrm{av}^{-1,p}$ norm, we have
\[
	\left| \int_{\mathbb{T}^n} (-\Delta)^{-s} u\cdot u_\mathrm{reg}(t)\; dx \right|
	\leq \left\| (-\Delta)^{-s} u \right\|_{\dot{W}_\mathrm{av}^{-1,p}}
	\left\| \nabla u_\mathrm{reg}(t) \right\|_{L^{p'}}.
\]
Invoking $L^{p'}$-$L^1$ estimate for the heat semigroup, we have
\[
	\left\| \nabla u_\mathrm{reg}(t) \right\|_{L^{p'}}
	= \| \nabla e^{t\Delta}u \|_{L^{p'}}
	\leq C_1 t^{-\frac{n}{2}\left(1-\frac{1}{p'}\right)} TV(u)
	= C_1 t^{-\frac{n}{2p}} TV(u).
\]
Thus
\[
	\left| \int_{\mathbb{T}^n} (-\Delta)^{-s} u\cdot u_\mathrm{reg}(t)\; dx \right|
	\leq C_1 t^{-n/2p}	
	\left\| (-\Delta)^{-s} u \right\|_{\dot{W}_\mathrm{av}^{-1,p}}
	TV(u).
\]

The estimate for the singular part is more involved.
 We proceed
\begin{align*}
	\left| \int_{\mathbb{T}^n} (-\Delta)^{-s} u\cdot u_\mathrm{sing}(t)\; dx \right|
	&= \left| \int_{\mathbb{T}^n} (-\Delta)^{-s} u\cdot \int_0^t  \Delta e^{\tau\Delta} u\; d\tau dx \right| \\
	&\leq \int_0^t \left| \int_{\mathbb{T}^n} (-\Delta)^{-s/2} u(-\Delta)^{1-s/2} e^{\tau\Delta} u\; dx \right| d\tau \\
	&\leq \int_0^t \left\| (-\Delta)^{-s/2} u \right\|_{L^2}
	\left\|(-\Delta)^{1-s/2} e^{\tau\Delta} u\right\|_{L^2}\; d\tau \\
	&= \|u\|_{\dot{H}_\mathrm{av}^{-s}} \int_0^t \left\| (-\Delta)^{1-s/2} e^{\tau\Delta} u \right\|_{L^2}\; d\tau.
\end{align*}
We note that
\[
	(-\Delta)^{1-s/2} e^{\tau\Delta} u 
	= -\sum_{j=1}^n (-\Delta)^{-s/2} \partial_j e^{\tau\Delta/2} e^{\tau\Delta/2} \partial_j u. 
\]
We use $L^2$-$L^2$ estimate for the heat semigroup to get
\[
	\left\|(-\Delta)^{-s/2} \partial_j e^{\tau\Delta/2} \right\|_{L^2\to L^2} \leq C\tau^{(s-1)/2}.
\]
Here we invoked the assumption that $s\leq1$.
 (This can be easily proved by the Parseval identity.)
 Thus,
\[
	\left\|(-\Delta)^{1-s/2} e^{\tau\Delta} u \right\|_{L^2} \leq C\tau^{(s-1)/2}
	\sum_{j=1}^n \left\| e^{\tau\Delta/2} \partial_j u \right\|_{L^2}.
\]
Using $L^2$-$L^1$ estimate for the heat semigroup, we end up with
\begin{align*}
	\left\|(-\Delta)^{1-s/2} e^{\tau\Delta} u \right\|_{L^2} 
	&\leq C' \tau^{(s-1)/2} \tau^{-n/4}
	\sum_{j=1}^n \| \partial_j u \|_{L^1} \\
	&\leq C'' \tau^{(2s-2-n)/4} TV(u).
\end{align*}
(To be precise, we take an approximate sequence of average-free smooth functions $f_k$ on $\mathbb{T}^n$ so that $f_k\to u$ in $L^2$ and $TV(f_k)\to TV(u)$.)
 We thus conclude that
\begin{align*}
	\left| \int_{\mathbb{T}^n} (-\Delta)^{-s} u\cdot u_\mathrm{sing}(t)\; dx \right| 
	&\leq C'' \|u\|_{\dot{H}_\mathrm{av}^{-s}} \int_0^t \tau^{(2s-2-n)/4}\; d\tau \, TV(u) \\
	&= C_2 \|u\|_{\dot{H}_\mathrm{av}^{-s}} TV(u) t^{(2+2s-n)/4},\quad
	C_2 = \frac{4}{2s+2-n} C'' 
\end{align*}
since $(2s-2-n)/4>-1$ by our assumption $n<2(s+1)$.

Combining the estimate for the regular part and the singular part, we have
\begin{equation} \label{ETS}
	\|u\|_{\dot{H}_\mathrm{av}^{-s}}^2 \leq \left( C_1 t^{-n/2p} \left\|(-\Delta)^{-s}u\right\|_{\dot{W}_\mathrm{av}^{-1,p}}
	+ C_2 t^{(2+2s-n)/4} \|u\|_{\dot{H}_\mathrm{av}^{-s}} \right) TV(u).
\end{equation}
We take $t$ so that the two terms in the right-hand side are balanced, i.e.,
\[
	C_1 t^{-n/2p} \left\|(-\Delta)^s u \right\|_{\dot{W}_\mathrm{av}^{-1,p}} 
	= C_2 t^{(2+2s-n)/4} \|u\|_{\dot{H}_\mathrm{av}^{-s}},
\]
or
\[
	t^\beta = \frac{C_1 \left\|(-\Delta)^s u \right\|_{\dot{W}_\mathrm{av}^{-1,p}}}{C_2 \|u\|_{\dot{H}_\mathrm{av}^{-s}}}, \quad
	\beta = \frac{1+s}{2} - \frac{n}{4} + \frac{n}{2p}.
\]
We fix this $t$ and observe that \eqref{ETS} becomes
\[
	\|u\|_{\dot{H}_\mathrm{av}^{-s}}^2
	\leq 2C_1 \left( \frac{C_1\left\|(-\Delta)^{-s} u\right\|_{\dot{W}_\mathrm{av}^{-1,p}}}{C_2 \|u\|_{\dot{H}_\mathrm{av}^{-s}}} \right)^{-n/2p\beta}
	\left\|(-\Delta)^{-s} u\right\|_{\dot{W}_\mathrm{av}^{-1,p}} TV(u)
\]
or
\[
	\|u\|_{\dot{H}_\mathrm{av}^{-s}}^{2-(n/2p\beta)}
	\leq C_3 \left\|(-\Delta)^{-s} u\right\|_{\dot{W}_\mathrm{av}^{-1,p}}^{1-(n/2p\beta)} TV(u)
\]
with $C_3=2C_1^{1-(n/2p\beta)} C_2^{n/2p\beta}$.
 If we take
\[
	\theta := \left(2-\frac{n}{2p\beta}\right)^{-1},
\]
we see that
\[
	1 - \frac{n}{2p\beta} = \frac{1}{\theta} -1.
\]
Thus
\[
	\|u\|_{\dot{H}_\mathrm{av}^{-s}}
	\leq C_* \left\|(-\Delta)^{-s} u\right\|_{\dot{W}_\mathrm{av}^{-1,p}}^{1-\theta} TV(u)^\theta
\]
with $C_*=C_3^\theta$.
 The definition of $\theta=\left(2-\frac{n}{2p\beta}\right)^{-1}$ is nothing but the scaling balance.
 The proof is now complete.
\end{proof}

We next prove a mild growth of $\left\|u(t)\right\|_{\dot{W}_\mathrm{av}^{-1,p}}$ as $t$ grows.
 For $\mathbb{T}^n=\prod_{i=1}^m(\mathbb{R}/\omega_i\mathbb{Z})$, we set $|\mathbb{T}^n|=\omega_i\cdots\omega_n$, which is the volume of a fundamental domain.
\begin{lemma} \label{LMG}
Let $u$ be a solution to \eqref{ETVP} (with $a\equiv1$) with initial datum $u_0\in\dot{H}_\mathrm{av}^{-s}(\mathbb{T}^n)$.
 Then,
\[
	\left\|(-\Delta)^{-s} u(t)\right\|_{\dot{W}_\mathrm{av}^{-1,p}}
	\leq |\mathbb{T}^n|^{1/p}t 
	+ \left\|(-\Delta)^{-s} u_0\right\|_{\dot{W}_\mathrm{av}^{-1,p}}
\]
for $1\leq p \leq\infty$.
\end{lemma}
\begin{proof}
We first observe that
\[
	\frac{d}{dt} \left\|(-\Delta)^{-s} u\right\|_{\dot{W}_\mathrm{av}^{-1,p}} (t)
	\leq \left\|(-\Delta)^{-s} u_t\right\|_{\dot{W}_\mathrm{av}^{-1,p}} (t)
\]
since $\frac{d}{dt}\|u\|_X\leq\|u_t\|_X$ by the triangle inequality of the norm $\|\cdot\|_X$.
 By Theorem \ref{TCPer}, we see
\[
	u_t = (-\Delta)^s \operatorname{div}Z
\]
with
\[
	\|Z\|_{L^{\infty}} \leq 1 \quad\text{and}\quad
	\left(u, -(-\Delta)^s \operatorname{div}Z \right)_{\dot{H}_\mathrm{av}^{-s}} = TV(u).
\]
Thus
\begin{multline*}
	\left\|(-\Delta)^{-s} u_t\right\|_{\dot{W}_\mathrm{av}^{-1,p}} (t)
	= \|\operatorname{div}Z\|_{\dot{W}_\mathrm{av}^{-1,p}} \\
	= \sup \left\{ \int_{\mathbb{T}^n} (-\nabla\varphi)\cdot Z\; dx \biggm|
	\|\nabla\varphi\|_{L^{p'}} \leq 1 \right\}
	\leq \|Z\|_{L^p} \leq |\mathbb{T}^n|^{1/p}.
\end{multline*}
We now conclude
\[
	\frac{d}{dt} \left\|(-\Delta)^{-s} u\right\|_{\dot{W}_\mathrm{av}^{-1,p}} (t)
	\leq |\mathbb{T}^n|^{1/p}
\]
which yields the desired inequality.
\end{proof}
We are now ready to prove our upper bound for the extinction time which is an easy extension of the case $s=1$.
\begin{thm} \label{TEEPer}
For $s\in(0,1]$, assume that $1\leq n\leq2(s+1)$, $1\leq p\leq\infty$.
 Assume that $1/2<\theta\leq1$ satisfies the scaling balance
\[
	s + \frac{n}{2} = (1-\theta) \left(2s+1+\frac{n}{p} \right)
	+ \theta(n-1).
\]
Then
\[
	T^*(u_0) \leq \frac{A_0}{a} \left\{ \left( 1+ \frac{aC_*^{1/\theta}\|u_0\|_{\dot{H}_\mathrm{av}^{-s}}^\gamma}{A_0^\gamma} \right)^{1/\gamma} -1 \right\}
\]
with $a:=|\mathbb{T}^n|^{1/p}$, $A_0:=\left\|(-\Delta)^{-s}u_0\right\|_{\dot{W}_\mathrm{av}^{-1,p}}$, $\gamma:=2-1/\theta$.
\end{thm}
\begin{proof}
We set $y(t)=\left\|u(t)\right\|_{\dot{H}_\mathrm{av}^{-s}}$ and recall the fundamental dissipation identity
\[
	(y^2/2)' = -TV(u).
\]
By the interpolation inequality (Lemma \ref{LInt}), we have
\[
	-TV(u) \leq -C_*^{-1/\theta} y(t)^{1/\theta}
	\left\|(-\Delta)^{-s}u(t) \right\|_{\dot{W}_\mathrm{av}^{-1,p}}^{1-1/\theta}.
\]
We may assume $y(t)\neq0$.
 We now apply our growth estimate to obtain
\[
	y(t)^{1-1/\theta} y'(t) 
	\leq -C_*^{-1/\theta} \left(|\mathbb{T}^n|^{1/p}t + \left\|(-\Delta)^{-s}u_0 \right\|_{\dot{W}_\mathrm{av}^{-1,p}}\right)^{1-1/\theta}
\]
since $1-1/\theta\leq0$.
 In other words,
\[
	\frac{1}{\gamma} \frac{d}{dt} y^\gamma
	\leq -C_*^{1/\theta} (at+A_0)^{\gamma-1}.
\]
Note that $\gamma\in(0,1]$ since $\theta$ satisfies $1/2<\theta\leq1$.
 Integrating both sides over $(0,t)$ we get
\[
	\frac{1}{\gamma} \left(y(t)^\gamma - y(0)^\gamma\right)
	\leq - \frac{C_*^{-1/\theta}}{a\gamma} \left\{(at+A_0)^\gamma - A_0^\gamma \right\}
\]
or
\begin{equation} \label{EEy}
	y(t)^\gamma \leq \|u_0\|_{\dot{H}_\mathrm{av}^{-s}}^\gamma
	- \frac{C_*^{-1/\theta}}{a} \left\{(at+A_0)^\gamma - A_0^\gamma \right\}.
\end{equation}
Since the right-hand side is nonnegative,
\[
	\frac{C_*^{-1/\theta}}{a} \left\{(at+A_0)^\gamma - A_0^\gamma \right\}
	\leq \|u_0\|_{\dot{H}_\mathrm{av}^{-s}}^\gamma
\]
or
\[
	at+A \leq A_0 \left( 1+ \frac{aC_*^{1/\theta}\|u_0\|_{\dot{H}_\mathrm{av}^{-s}}^\gamma}{A_0^\gamma} \right)^{1/\gamma}.
\]
Thus, the desired estimate follows from \eqref{EEy}.
 (Note that if $\theta=1/2$, $\gamma$ must be zero so the above argument does not apply.
 The case $\theta=1/2$ corresponds to the case $1\leq n<2(s+1)$ and $p=\infty$ as we observed before.)
\end{proof}

{\BLUE
\section{Regularity} \label{S6} 

A fundamental feature of total variation flows, that sets them apart from usually considered quasilinear parabolic equations, is that $u(t)$ is typically only a $BV$ function for $t>0$. That is, the distributional derivative $D u$ is in general not an integrable function, but a vector measure. In particular, as we have seen in Section \ref{S4}, $u$ can have jump discontinuities along hypersurfaces. This is a desirable feature from the point of view of applications to image processing, where jumps corresponds to sharp contours in images. 

Rigorously speaking, a point $x \in \Omega$ is called an \emph{(approximate) jump point} of $w \in L^1_{loc}(\Omega)$ if there exist real numbers $w^- = w^-(x)$, $w^+=w^+(x)$, $w^- \neq w^+$ and a vector $\nu_w = \nu_w(x)$ such that 
\begin{equation} \label{jump_def}
\lim_{r \to 0^+} \fint_{B_r^-(x, \nu_w)}|w(y) - w^-| dy =0, \quad \lim_{r \to 0^+} \fint_{B_r^+(x, \nu_w)}|w(y) - w^+| dy =0 
\end{equation} 
where the symbol $\fint$ denotes average integral over a set and $B_r^\pm(x, \nu_w)$ are the half-balls
\[B_r^-(x, \nu_w) = \left\{y \in B_r(x) \bigm| (y-x)\cdot \nu_w \geq 0\right\}, \quad B_r^+(x, \nu_w) = \left\{y \in B_r(x) \bigm| (y-x)\cdot \nu_w \leq 0\right\}.   \]
To be precise, the triple $(w^+(x), w^-(x), \nu_w(x))$ is defined up to permutation of $w^+(x)$ and $w^-(x)$ with simultaneous multiplication of $\nu_w(x)$ by $-1$. We also recall the notion of approximate continuity, closely related to Lebesgue points. As in \cite{AFP}, we say that $x \in \Omega$ is a \emph{point of approximate continuity} of $w$, if there exists $w^*=w^*(x) \in \mathbb{R}$ such that
\begin{equation} \label{approx_cont}
\lim_{r \to 0^+} \fint_{B_r(x)}|w(y) - w^*| dy =0.   
\end{equation}
The set of jump points of $w$ is denoted by $J_w$, while the set of \emph{approximate discontinuity}, i.e., the complement of the set of points of approximate continuity of $w$, is denoted by $S_w$. Clearly $J_w \subset S_w$. By the Federer--Vol'pert theorem \cite[Theorem 3.78]{AFP}, if $w \in BV(\Omega)$, then $S_w$ (and $J_w$) is \emph{countably $\mathcal{H}^{n-1}$-rectifiable}, in particular, it can be covered by a countable sum of graphs of $C^1$ functions up to a $\mathcal{H}^{n-1}$-negligible set (i.e.~a set of $\mathcal{H}^{n-1}$ measure $0$) \cite[p.~80]{AFP}. Moreover, $J_w$ coincides with $S_w$ up to a $\mathcal{H}^{n-1}$-negligible set, that is $\mathcal{H}^{n-1}(S_w \setminus J_w) =0$. We also note that if \eqref{jump_def} holds for a given $x \in \Omega$ with $w^- = w^+$, then also \eqref{approx_cont} holds with $w^* = w^\pm$. On the other hand, if \eqref{approx_cont} holds with  $w^* \in \mathbb{R}$, then both equations in \eqref{jump_def} are satisfied with $w^+ = w^- = w^*$. Thus, if $w \in BV(\Omega)$, then $w^\pm(x)$ can be defined for $\mathcal{H}^{n-1}$-a.e.~$x \in \Omega$. Furthermore, one-sided traces of $w$ along any $C^1$ hypersurface in $\Omega$ are well defined \cite[Theorem 3.77]{AFP} and coincide with $w^\pm$ up to a pointwise permutation \cite[Remark 3.79]{AFP}. 

In this section we denote by $\nabla w$ the Radon--Nikodym derivative of $Dw$ with respect to the Lebesgue measure $\mathcal{L}^n$. Thus, we have 
\[ Dw = \nabla w \mathcal{L}^n + D^s w, \]
where $D^s w$ is called the singular part of $Dw$. The singular part can be further decomposed into the jump part $D^j w = (w^+-w^-)\nu_w \, \mathcal{H}^{n-1}$ supported on $J_w$ and the remaining Cantor part $D^c w$, see \cite{AFP} for details. 

It is then natural to ask about the location of the jumps of $u$, or more generally, of (the support of) $D^s u$. The situation is particularly simple in the 1D setting (in order to avoid technicalities related to the boundary, we restrict ourselves to the periodic case $\Omega = \mathbb{R}/\mathbb{Z}$). In this case $BV$ functions are well approximated by step functions. Indeed, given $w \in BV(\Omega)$ and a natural number $N$, we define $w^N$ by 
\[w^N(x) =  N \int_{(k-1)/N}^{k/N} w(y) dy  \text{ for } x \in [(k-1)/N, k/N), \ k = 1, \ldots, N. \]
If $w \in C^1(\overline{\Omega})$, then $w$ attains its average value $N \int_{(k-1)/N}^{k/N} w(y)dy$ on $[(k-1)/N, k/N)$ and, by the fundamental theorem of calculus, we have for $x \in [(k-1)/N, k/N)$
\begin{equation}\label{bv_avg_ineq}
|w(x) - w^N(x)| = \left|w(x) - N \int_{(k-1)/N}^{k/N} w(y)dy\right| \leq \int_{[(k-1)/N, k/N)} |D w|.
\end{equation} 
It is not difficult to see that the same holds for general $w \in BV(\Omega)$, $\Omega = (a,b)$, if we identify it with its left-continuous representative which is of form 
\[w(x) = c + \int_{(a,x)} Dw \]
with some $c \in \mathbb{R}$, see \cite[Thm.\ 3.28]{AFP}. From \eqref{bv_avg_ineq} we deduce 
\begin{multline*} 
\|w - w^N\|_{L^2(\Omega)}^2 = \sum_{k=1}^N \int_{(k-1)/N}^{k/N} \left|w(x) - N \int_{(k-1)/N}^{k/N} w(y)dy\right|^2dx \leq \frac{1}{N}\sum_{k=1}^N \left(\int_{[k-1)/N,k/N)} |Dw|\right)^2 \\ \leq \frac{1}{N}  \left( \sum_{k=1}^N \int_{[k-1)/N,k/N)} |Dw|\right)^2 = \frac{1}{N} TV(w)^2,
\end{multline*} 
which clearly tends to $0$ as $N \to \infty$, i.e.~$w^N \to w$ in $L^2(\Omega)$.  

Moreover, again identifying $w$ with the left-continuous representative,
\begin{multline*}TV(w^N) = \int_\Omega |D w^N| =  \sum_{k=1}^N \left|N\int_{k/N}^{(k+1)/N} w(y) dy - N\int_{(k-1)/N}^{k/N} w(y) dy\right| \\= N \sum_{k=1}^N \left|\int_{(k-1)/N}^{k/N} w(y + 1/N) - w(y) dy \right| = N \sum_{k=1}^N \left|\int_{(k-1)/N}^{k/N} \left(\int_{[y,y + 1/N)} Dw\right) dy\right| 
\\ \leq N \sum_{k=1}^N \int_{(k-1)/N}^{k/N} \left(\int_{[y,y + 1/N)} |Dw|\right) dy =  \int_\Omega N  \left(\int_{[y,y + 1/N)} |Dw|\right) dy = \int_\Omega |Dw|,   
\end{multline*}
where the intuitive last equality can be justified using Fubini's theorem. Thus, by lower semicontinuity of the total variation on any open subset of $\Omega$ and \cite[Prop.\ 1.80]{AFP}, measures $|Dw^N|$ converge weakly star to $|Dw|$. 

We apply this approximation to the initial datum $u_0$, obtaining step functions $u_0^N$. Let $u^N$ be the solution to the total variation flow with initial datum $u_0^N$. Then $u^N(t)$ is a step function for $t>0$. Precisely, between merging times, the solution is given by formulae \eqref{step_evo}--\eqref{theta_k}. We observe that $a^k$ is a decreasing (resp.\ increasing) function of time if and only if $a^k > a^{k-1}$ and $a^k > a^{k+1}$ (resp.\ $a^k < a^{k-1}$ and $a^k < a^{k+1}$). Thus, the functions $|a^{k+1}-a^k|$ are decreasing for $k=1,\ldots,m$. Since 
\[|Du^N(t)| = \sum_{k=1}^m |a^{k+1}(t)-a^k(t)|\delta_{x_k},\]
this implies $|Du^N(t)| \leq |D u^N_0|$ as measures, which is equivalent to saying that 
\begin{equation} \label{test_ineq}
 \int_\Omega \varphi \,d |Du^N(t)| \leq \int_\Omega \varphi \,d |Du^N_0| \quad \text{for any } \varphi \in C_c(\Omega), 
\end{equation} 
in other words $TV_\varphi(u^N(t))\leq TV_\varphi(u(t))$ for $\varphi \in C_c(\Omega)$. 

We now want to pass to the limit $N \to \infty$ in \eqref{test_ineq}. Since $u_0^N \to u_0$ in $L^2(\Omega)$, we have, by monotonicity of $-\partial TV$, $u^N(t) \to u(t)$ in $L^2(\Omega)$ for $t>0$. By lower semicontinuity of $TV_\varphi$, we obtain 
\[ \liminf_{N \to \infty} \int_\Omega \varphi \,d |Du^N(t)| \geq \int_\Omega \varphi \,d |Du(t)|.\]
On the other hand, we have showed that measures $|Du_0^N|$ converge weakly star to $|Du_0|$, i.e. 
\[\int_\Omega \varphi \,d |Du_0^N| \to \int_\Omega \varphi \,d |Du_0|.\] 
Thus, we have 
\begin{thm} \label{thm:Du_bound}
    Let $\Omega$ be an interval and suppose that $u$ is the solution to the total variation flow with initial datum $u_0 \in BV(\Omega)$. Then 
    \begin{equation}\label{meas_ineq} |Du(t)| \leq |Du_0| \quad \text{as measures for } t>0.  
    \end{equation} 
\end{thm}
Inequality \eqref{meas_ineq} implies that $|\nabla u(t)| \leq |\nabla u_0|$ Lebesgue-almost everywhere in $\Omega$ and $|D^s u(t)| \leq |D^s u_0|$ as measures. This has quite strong consequences in terms of regularity. In particular, it implies that the total variation flow preserves Sobolev spaces $W^{1,p}(\Omega)$ for $p \in [1, \infty]$, or the space $SBV(\Omega)$ of functions $w \in BV(\Omega)$ such that $D^c w = 0$. Many of these consequences were derived by a technique similar to the one presented here in \cite{BF}. Independently, a generalization of Theorem \ref{thm:Du_bound} was obtained in \cite{BCNO} by a different technique. We note that one can also show that the 1D total variation flow  preserves second-order $BV$ regularity. That is, if $D u_0 \in BV(\Omega)$, then $D u(t) \in BV(\Omega)$ for $t>0$ \cite{MR}. This is the highest regularity preserved by the flow: in fact, if $u_0 \in C^\infty(\overline{\Omega})$ is non-monotone, then $Du(t)$ will have jump discontinuities for any $t>0$ small enough, see e.g.\ \cite{KMR}.  

Theorem \ref{thm:Du_bound} can also be generalized to the vector-valued setting, with total variation calculated with respect to possibly non-Euclidean norms \cite{GiL, GrL}. If the total variation is replaced by a non-homogeneous functional, the pointwise inequality $|\nabla u(t)| \leq |\nabla u_0|$ in general fails. However one can still obtain a bound on $D^s u$ \cite{MS, GrL}. 

As for the total variation flow in higher-dimensional domains, the pointwise estimate on $\nabla u$ is known to be violated, as evidenced by examples of \emph{facet bending} \cite{ACC}. However one can still show an estimate on the size of jumps of $u(t)$ in terms of $u_0$. Known results of this type rely on the minimizing movements approximation briefly discussed in Section \ref{S2S1}, consisting in iteratively solving a minimization problem for the functional $\mathcal{E}^\lambda_f$ given by \eqref{min_mov}, which in the case $\mathcal{E}=TV$ reads 
\begin{equation} \label{rof_min} 
TV^\lambda_f(w) = \lambda TV(w) + \frac{1}{2}\int_\Omega (w-f)^2 dx.
\end{equation} 
Initial works on this subject used an argument based on the fact that level sets of the minimizer of $TV^\lambda_f$ solve a prescribed mean curvature problem \cite{CCN, CJN}. This has been later significantly generalized in \cite{Val, CL} by different techniques, without reference to particular structure of $\mathcal{E}$ or level sets of the minimizer, thus allowing to handle also the vector-valued case. In fact, it is enough to assume that $\mathcal{E}$ has a mild regularity property of \emph{differentiability along inner variations}, which holds for $TV$ \cite[Chapter 10]{Giu}. That is, setting $w_\varphi^\tau(x) = w(x + \tau \varphi(x))$ for $w \in BV(\Omega)$, $\varphi \in C^\infty_c(\Omega)^n$, $x \in \Omega$, $\tau \in \mathbb{R}$, the real function $R_\varphi\colon\tau \mapsto TV(w_\varphi^\tau)$ is differentiable for any $\varphi \in C^\infty_c(\Omega)^n$. To obtain the desired assertion, we actually need to use a bit more complicated mixed variations $w_\varphi^{\tau, \vartheta} := (1 - \vartheta) w + \vartheta w_\varphi^\tau$. By convexity of $TV$ we have 
\[\tfrac{1}{\tau}(TV(w_\varphi^{\tau, \vartheta})-TV(w)) \leq \tfrac{1}{\tau}((1-\vartheta)TV(w) + \vartheta TV(w_\varphi^\tau)- TV(w)) = \tfrac{\vartheta}{\tau}(R_\varphi(\tau)- R_\varphi(0))\to \vartheta R_\varphi'(0), \]
\[\tfrac{1}{\tau}(TV(w_\varphi^{-\tau, \vartheta})-TV(w)) \leq \tfrac{1}{\tau}((1-\vartheta)TV(w) + \vartheta TV(w_\varphi^{-\tau})- TV(w)) = \tfrac{\vartheta}{\tau}(R_\varphi(-\tau)- R_\varphi(0))\to -\vartheta R_\varphi'(0) \]
as $\tau \to 0^+$. Thus, 
\begin{equation} \label{mixed_var_ineq} 
\limsup_{\tau \to 0^+} \tfrac{1}{\tau}(TV(w_\varphi^{\tau, \vartheta})-TV(w)) + \limsup_{\tau \to 0^+} \tfrac{1}{\tau}(TV(w_\varphi^{-\tau, \vartheta})-TV(w)) \leq 0.
\end{equation} 
Since the function $\tau \to TV(w_\varphi^{\tau, \vartheta})$ is convex, it actually follows from this inequality that it is also differentiable (at least at $0$). 

Now, let $v \in BV(\Omega)$ be a minimizer of $TV_f^\lambda$ for a given $f \in BV(\Omega)$, $\lambda >0$. We assume moreover that $f \in L^\infty(\Omega)$, in which case it is easy to show that $v \in L^\infty(\Omega)$ and $\|v\|_{L^\infty(\Omega)} \leq \|f\|_{L^\infty(\Omega)}$. Since $v$ is a minimizer, we have
\[\liminf_{\tau \to 0^+} \tfrac{1}{\tau}(TV^\lambda_f(v_\varphi^{\tau, \vartheta})-TV^\lambda_f(v)) \geq 0, \quad \liminf_{\tau \to 0^+} \tfrac{1}{\tau}(TV^\lambda_f(v_\varphi^{-\tau, \vartheta})-TV^\lambda_f(v)) \geq 0. \]
Adding the two inequalities together and taking into account \eqref{mixed_var_ineq}, we deduce 
\begin{equation} \label{fid_ineq}
0 \leq \limsup_{\tau \to 0^+} \frac{1}{2 \tau}\int_\Omega (v_\varphi^{\tau, \vartheta} - f)^2  - (v - f)^2dx + \limsup_{\tau \to 0^+} \frac{1}{2 \tau}\int_\Omega (v_\varphi^{-\tau, \vartheta} - f)^2 - (v - f)^2 dx. 
\end{equation} 
Further on, we will only use \eqref{fid_ineq}. Note that it does not involve $TV$ at all. We calculate 
\[(v_\varphi^{\pm\tau, \vartheta} - f)^2 - (v - f)^2 = (v_\varphi^{\pm\tau, \vartheta} - v)(v + v_\varphi^{\pm\tau, \vartheta}  - 2f) = \vartheta (v_\varphi^{\pm\tau} - v)((2-\vartheta) v + \vartheta v_\varphi^{\pm\tau} - 2f).\]
Thus, dividing \eqref{fid_ineq} by $\vartheta$,  
\begin{equation}\label{fid_ineq2} 
0 \leq \limsup_{\tau \to 0^+} \frac{1}{2 \tau}\int_\Omega (v_\varphi^{\tau} - v)((2-\vartheta) v + \vartheta v_\varphi^{\tau} - 2f)dx + \limsup_{\tau \to 0^+} \frac{1}{2 \tau}\int_\Omega (v_\varphi^{-\tau} - v)((2-\vartheta) v + \vartheta v_\varphi^{-\tau} - 2f)dx. 
\end{equation} 

All the calculations so far were done for arbitrary $\varphi \in C_c^\infty(\Omega)$. Now let $\Gamma \subset \Omega$ be any one of the $C^1$ graphs that cover $J_v$ by rectifiability, let $x_0 \in \Gamma$, and let $\nu_0$ be a vector normal to $\Gamma$ at $x_0$. By isometric change of coordinates, we can assume without loss of generality that $x_0=0$, $\nu_0=(0,\ldots,0,1)$ and $\Gamma \supset \{ (x', x_n) \colon x' \in B^{n-1}_r,\ x_n = \gamma(x')\}$ where $\gamma \in C^1(B^{n-1}_r)$ and $B^{n-1}_r := \{x' \in \mathbb{R}^{n-1} \colon |x'| < r\}$ for $r>0$. Owing to differentiability of $\gamma$, possibly decreasing $r$ we can assume that $\gamma(B^{n-1}_s) \subset (-s/2,s/2)$ for $0<s\leq r$. Then, we take $\varphi = \nu_0 \psi = (0, \ldots,0, \psi)$ supported in $Q_r := B^{n-1}_r \times (-r,r)$ with $\psi \in C_c^\infty(Q_r)$ such that $0\leq \psi \leq 1$ and $\psi(x', x_n) = \psi_\parallel(x') \psi_\perp(x_n)$, $\psi_\parallel=1$ on $B^{n-1}_{r-\varepsilon}$, $\psi_\perp = 1$ on $(-r+\varepsilon, r-\varepsilon)$ for a temporarily fixed $\varepsilon \in (0, r/2)$.  

For $\tau \in (0, r/2)$ we rewrite 
\begin{multline}\label{fid_split}  
\frac{1}{2 \tau}\int_\Omega (v_\varphi^{\tau} - v)((2-\vartheta) v + \vartheta v_\varphi^{\tau} - 2f)dx = \frac{1}{2 \tau}\int_{B^{n-1}_r}\int_{\gamma(x') - \tau \psi_\parallel(x')}^{\gamma(x')} (v_\varphi^{\tau} - v)((2-\vartheta) v + \vartheta v_\varphi^{\tau} - 2f)dx_n dx'
 \\+ \frac{1}{2 \tau} \int_{B^{n-1}_r}\int_{(-r,r)\setminus[\gamma(x') - \tau \psi_\parallel(x'),\gamma(x')]} (v_\varphi^{\tau} - v)((2-\vartheta) v + \vartheta v_\varphi^{\tau} - 2f) dx_n dx' .
  \end{multline} 
  Identifying $v$ with its precise representative, the slicing properties of $BV$ functions~\cite[\S 3.11, Theorem 3.107]{AFP} ensure that for $\mathcal{L}^{n-1}$-a.e.~$x' \in B^{n-1}_r$ the function $v_{x'} \colon x_n\mapsto v(x',x_n)$ is in $BV((-r,r))$, and we can write for 
  $\mathcal{L}^1$-a.e.~$x_n \in (- r, r)$:
  \[v_\varphi^{\tau}(x',x_n) - v(x', x_n) = v(x',x_n + \tau \psi(x',x_n)) - v(x', x_n) =  \int_{(x_n, x_n + \tau \psi(x',x_n))} Dv_{x'} .\]
  Since the support of $\psi$ is contained in $Q_r$, for small enough $\tau>0$ we have $x_n + \tau \psi(x',x_n) \leq \min(x_n + \tau \psi_\parallel(x'),r)$ for $x
_n \in (-r,r)$. Therefore
    \begin{multline*}\left|(v_\psi^{\tau} - v)((2-\vartheta) v + \vartheta v_\varphi^{\tau} - 2f)\right|(x', x_n) 
      \leq  4 \|f\|_{L^\infty(\Omega)} \left|v_\psi^{\tau} - v\right|(x', x_n) \\ 
         \leq 4 \|f\|_{L^\infty(\Omega)} \int_{(x_n, x_n + \tau \psi(x',x_n))} |Dv_{x'}|  \leq 4 \|f\|_{L^\infty(\Omega)} \int_{(x_n, \min(x_n + \tau \psi_\parallel(x'),r))} |Dv_{x'}| 
\end{multline*}
and so, by Fubini's theorem, for $\mathcal{L}^{n-1}$-a.e.~$x' \in B^{n-1}_r$, 
 \begin{multline} \label{fid_fub} 
		\frac{1}{\tau}\int_{(-r,r)\setminus[\gamma(x') - \tau \psi_\parallel(x')),\gamma(x')]}\left|(v_\psi^{\tau} - v)((2-\vartheta) v + \vartheta v_\varphi^{\tau} - 2f)\right|(x', x_n)\, d x_n
  \\  \leq 4 \|f\|_{L^\infty(\Omega)} \frac{1}{\tau} \int\int \mathbf{1}_{(-r,r)\setminus[\gamma(x') - \tau \psi_\parallel(x')),\gamma(x')]}(x_n) \mathbf{1}_{(x_n,\min(x_n + \tau \psi_\parallel(x'),r))}(s)\, d \big|Dv_{x'}\big|(s)\, d x_n 
  \\ \leq 4\|f\|_{L^\infty(\Omega)} \frac{1}{\tau} \int\int \mathbf{1}_{(s - \tau \psi_\parallel(x'),s))}(x_n) \mathbf{1}_{(-r,r)\setminus\{\gamma(x')\}}(s) \, d x_n \, d \big|Dv_{x'}\big|(s) \\ \leq 4\|f\|_{L^\infty(\Omega)} \psi_\parallel(x') \int_{(-r,r)\setminus\{\gamma(x')\}} |Dv_{x'}| \leq 4\|f\|_{L^\infty(\Omega)} \int_{(-r,r)\setminus\{\gamma(x')\}} |Dv_{x'}|.
	\end{multline}
Appealing to \cite[Theorem 3.107]{AFP},
\begin{equation} \label{fid_rem_est}
\frac{1}{2 \tau} \int_{B^{m-1}_r}\int_{(-r,r)\setminus[\gamma(x') - \tau \psi_\parallel(x'),\gamma(x')]} (v_\varphi^{\tau} - v)((2-\vartheta) v + \vartheta v_\varphi^{\tau} - 2f) dx_n dx' \leq 2 \|f\|_{L^\infty(\Omega)}\int_{Q_r \setminus \Gamma} |Dv|.
\end{equation}
Thus we have estimated the second term on the r.h.s.~of \eqref{fid_split}. 

As for the other one, using \cite[Theorem 3.108]{AFP}, for $\mathcal{L}^{n-1}$-a.e.\ $x' \in B^{m-1}_r$ we have 
  \begin{equation*}
    \frac{1}{\tau}\int_{\gamma(x') - \tau \psi_\parallel(x')}^{\gamma(x')} (v_\psi^{\tau} - v)((2-\vartheta) v + \vartheta v_\varphi^{\tau} - 2f)dx_n \to \psi\,(v^+ - v^-)((2-\vartheta) v^- + \vartheta v^+ - 2f^-)\bigg|_{(x',\gamma(x'))}, 
  \end{equation*} 
  where $v^-$, $f^-$ (resp.~$v^+$, $f^+$) are the approximate limits corresponding to traces of $v$, $f$ along $\Gamma$ "from below" (resp.~"from above"). By a rough estimate in the vein of \eqref{fid_fub}, we can show that  
\[\frac{1}{\tau}\int_{\gamma(x') - \tau \psi_\parallel(x')}^{\gamma(x')}\left|(v_\psi^{\tau} - v)((2-\vartheta) v + \vartheta v_\varphi^{\tau} - 2f)\right|(x', x_n)\, d x_n
  \leq  4\|f\|_{L^\infty(\Omega)} \int_{(-r,r)} |Dv_{x'}|.\]
Since the r.h.s.~is an integrable function of $x'$ (again by \cite[Theorem 3.107]{AFP}), we can apply dominated convergence theorem to show that the first term on the r.h.s.~of \eqref{fid_split} converges as $\tau \to 0^+$ to   
\begin{equation*}
    \frac{1}{2} \int_{B_r^{n-1}}\psi\,(v^+ - v^-)((2-\vartheta) v^- + \vartheta v^+ - 2f^-)\bigg|_{(x',\gamma(x'))}dx'. 
  \end{equation*} 
  Thus, recalling \eqref{fid_rem_est}, 
\begin{multline*} 
\limsup_{\tau \to 0^+} \frac{1}{2 \tau}\int_\Omega (v_\varphi^{\tau} - v)((2-\vartheta) v + \vartheta v_\varphi^{\tau} - 2f)dx \\ \leq \frac{1}{2} \int_{B_r^{n-1}}\psi\,(v^+ - v^-)((2-\vartheta) v^- + \vartheta v^+ - 2f^-)\bigg|_{(x',\gamma(x'))}dx' + 2 \|f\|_{L^\infty(\Omega)}\int_{Q_r \setminus \Gamma} |Dv|.
\end{multline*} 
Repeating the same reasoning, we also obtain 
\begin{multline*} 
\limsup_{\tau \to 0^+} \frac{1}{2 \tau}\int_\Omega (v_\varphi^{-\tau} - v)((2-\vartheta) v + \vartheta v_\varphi^{-\tau} - 2f)dx \\ \leq \frac{1}{2} \int_{B_r^{n-1}}\psi\,(v^- - v^+)((2-\vartheta) v^+ + \vartheta v^- - 2f^+)\bigg|_{(x',\gamma(x'))}dx' + 2 \|f\|_{L^\infty(\Omega)}\int_{Q_r \setminus \Gamma} |Dv|.
\end{multline*}
Summing these two inequalities, recalling \eqref{fid_ineq2}, and passing with $\vartheta \to 0^+$, $\varepsilon \to 0^+$, 
\[
0 \leq \int_{B_r^{n-1}}(v^+ - v^-)(v^- - v^+ + f^+ - f^-)\bigg|_{(x',\gamma(x'))}dx' + 4 \|f\|_{L^\infty(\Omega)}\int_{Q_r \setminus \Gamma} |Dv|.
\]
Finally, we divide the obtained inequality by $|B_r^{n-1}|$ and pass to the limit $r \to 0^+$. By \cite[eq.~(2.41) on p.~79]{AFP}, the second term vanishes in the limit for $\mathcal{H}^{n-1}$-a.e.~$x_0 \in \Gamma$. Here $\mathcal{H}^{n-1}$ denotes the $n-1$-dimensional Hausdorff measure, whose restriction to $\Gamma$ corresponds to the classical surface measure on $\Gamma$. Since also $\mathcal{H}^{n-1}$-a.e.~$x_0 \in \Gamma$ is a Lebesgue point of the function $(v^+ - v^-)(v^- - v^+ + f^+ - f^-)$ with respect to the surface measure, we obtain 
\[0 \leq (v^+ - v^-)(v^- - v^+ + f^+ - f^-)\]
for $\mathcal{H}^{n-1}$-a.e.~$x_0 \in \Gamma$ (see \cite{CL} for a more detailed explanation of this part). 
We deduce 
\[|v^+ - v^-| \leq |f^+ - f^-| \quad \text{for } \mathcal{H}^{n-1}\text{-a.e. } x_0 \in J_v.\]
In particular, $J_v$ is contained in $J_f$ up to a $\mathcal{H}^{n-1}$-negligible set. 

Applying this to the minimizing movements approximation, if $u_0 \in BV(\Omega) \cap L^\infty(\Omega)$, we have $J_{u^N(t)} \subset J_{u_0}$ for $t >0$, $N \in \mathbb{N}$ up to a $\mathcal{H}^{n-1}$-negligible set and 
\begin{equation} \label{min_mov_est}  
|u^N(t)^+ - u^N(t)^-| \leq |u_0^+ - u_0^-| \quad \text{for } \mathcal{H}^{n-1}\text{-a.e. } x_0 \in J_{u_0}.
\end{equation} 
Recall that for $v \in BV(\Omega)$, the functions $v^\pm$ can be defined $\mathcal{H}^{n-1}$-a.e.~in $\Omega$ by setting $v^+=v^-$ to be the approximate limit of $v$ at any point where it exists. Indeed, by the Federer--Vol'pert theorem, $\mathcal{H}^{n-1}$-a.e.~point in $\Omega \setminus J_v$ is a point of approximate continuity of $v$ \cite{AFP}. 

It remains to transfer the result to the total variation flow. We will follow the approach from \cite{CCN,CJN} relying on the theory of accretive operators on Banach spaces, see Appendix A in \cite{ACM} and references therein. We will focus on some details related to limit passage with the bound on jump size \eqref{min_mov_est} which seem to be omitted in \cite{CCN,CJN}. We define $A_\infty$ as the restriction of $\partial_{L^2} TV$ to $L^\infty(\Omega)$. Using the characterization of $\partial_{L^2} TV$ in terms of Cahn--Hoffman vector fields (Theorem \ref{TSub}), it is not difficult to see that it is \emph{accretive}, i.e.,
\[\|v_1 -v_2\|_{L^\infty(\Omega)} \leq \|v_1 -v_2 + \lambda(w_1 - w_2)\|_{L^\infty(\Omega)} \quad \text{for } v_1, v_2 \in D(A_\infty),\ w_1 \in A_\infty(v_1),\  w_2 \in A_\infty(v_2),\ \lambda >0,\]
where $D(A_\infty)$ denotes the set of $v \in L^\infty(\Omega)$ such that $A_\infty(v)$ is non-empty. One can show it by approximating $\infty$ with finite $p$, using convexity of the power function and integrating by parts. Moreover, $A_\infty$ satisfies the \emph{range condition}  
\[\overline{D(A_\infty)} \subset R(I + \lambda A_\infty) \quad \text{for all } \lambda >0,\]
where $R(I + \lambda A_\infty)$ is the range of the operator $I + \lambda A_\infty$. In fact $R(I + \lambda A_\infty) = L^\infty(\Omega)$ since minimizers of \eqref{min_mov} with $f\in L^\infty(\Omega)$ belong to $L^\infty(\Omega)$. Thus, by the Crandall--Liggett theorem \cite[Theorem A.28]{ACM}, if $u_0 \in \overline{D(A_\infty)}$,
the minimizing movements approximation $u^N(t)$ defined in \eqref{min_mov_approx} converges to $u(t)$ in $L^\infty(\Omega)$ for $t > 0$. 

Recall that $u^N$ is piecewise constant as a function from $[0, \infty)$ to $L^2(\Omega)$, in particular its image is countable. Thus, there is a $\mathcal{H}^{n-1}$-full subset $J$ of $J_{u_0}$ whose elements are either jump points or points of approximate continuity of $u^N(t)$ for all $N \in \mathbb{N}$, $t>0$. Moreover, since the jump sets of $u^N$ are $\mathcal{H}^{n-1}$-almost contained in $J_{u_0}$, they can all be $\mathcal{H}^{n-1}$-almost covered by the same countable family of $C^1$ surfaces. Owing to this, we can assume that in case $x \in J$ is a jump point for $u^N(t)$, $N \in \mathbb{N}$, $t>0$, then $\nu_{u^N(t)}(x) = \nu_{u_0}(x)$. Therefore we have 
\begin{equation} \label{J_lim}
\lim_{r \to 0^+} \fint_{B^{\pm}_r(x, \nu_{u_0})} |u^N(t,y) - u^N(t)^\pm(x)|dy = 0
\end{equation} 
for all $x \in J$, $N \in \mathbb{N}$, $t>0$. By the uniform bound $\|u^N(t)\|_{L^\infty(\Omega)} \leq \|u_0\|_{L^\infty(\Omega)}$, the sequences $\{u^N(t)^\pm(x)\}_{N \in \mathbb{N}}$ are bounded for $x \in J$ and thus each one has a limit point $a^\pm$. Given $\varepsilon >0$, 
\begin{multline} \label{jump_triangle}
\fint_{B^{\pm}_r(x, \nu_{u_0})} |u(t,y) - a^\pm|dy \\ \leq \fint_{B^{\pm}_r(x, \nu_{u_0})} |u(t,y) - u^N(t,y)|dy + \fint_{B^{\pm}_r(x, \nu_{u_0})} |u^N(t,y) - u^N(t)^\pm(x)|dy + |u^N(t)^\pm(x) - a^\pm| \\ \leq \fint_{B^{\pm}_r(x, \nu_{u_0})} |u^N(t,y) - u^N(t)^\pm(x)|dy + 2 \varepsilon
\end{multline} 
for $N$ large enough, independently of $r>0$. As $\varepsilon >0$ is arbitrarily small, we deduce from \eqref{J_lim} 
\begin{equation*}
\lim_{r \to 0^+} \fint_{B^{\pm}_r(x, \nu_{u_0})} |u(t,y) - a^\pm|dy = 0, 
\end{equation*} 
hence $x \in J_{u(t)} \cup (\Omega \setminus S_{u(t)})$ and, by \eqref{min_mov_est},  
\[ |u^+(t) - u^-(t)| = |a^+ - a^-| \leq |u_0^+ - u_0^-|.\] 
By an estimate similar to \eqref{jump_triangle}, using the inclusion $\mathcal{H}^{n-1}$-almost inclusion $S_{u^N(t)} \subset S_{u_0}$, we also show that $S_{u(t)} \subset S_{u_0}$ (equivalently, $J_{u(t)} \subset J_{u_0}$) up to a $\mathcal{H}^{n-1}$-negligible set. 

Finally, we would like to remove the enigmatic assumption $u_0 \in \overline{D(A_\infty)}$. Instead, let us only assume that $u_0 \in BV(\Omega) \cap L^n(\Omega)$. By the $L^n$-$L^\infty$ regularization property \cite[Theorem 2.16]{ACM} and the pointwise estimate $|u_t(t)| \leq 2 \frac{|u(s)|}{t-s}$ for $0 \leq s < t$ following from homogeneity of $TV$ \cite[(2.34)]{ACM}, we have $u(t) \in  D(A_\infty)$ for $t>0$. Thus, by the previous step, $J_{u(t)} \subset J_{u(s)}$ and the inequality $|u^+(t) - u^-(t)| \leq |u^+(s) - u^-(s)|$ holds $\mathcal{H}^{n-1}$-a.e.~in $\Omega$ for $0 < s < t$. Let us take a sequence of positive numbers $s_k$ such that $s_k \to 0$ as $k \to \infty$. By lower semicontinuity of $TV$ and its monotonicity along trajectories of the flow, 
\[\lim_{k\to \infty} \int_\Omega |D u(s_k)| = \int_\Omega |D u_0|. \]
This improves the convergence $u(s_k) \to u_0$ in $L^2(\Omega)$ to strict convergence in $BV(\Omega)$, whence $|Du(s_k)|$ converges to $|Du_0|$ weakly* in the space of signed Radon measures $M(\Omega)$ \cite[Proposition 1.80]{AFP}. Thus, using \cite[Proposition 1.63]{AFP}, given $x \in \Omega$, 
\[ \int_{B_r(x)} |Du(s_k)| \to \int_{B_r(x)} |Du_0| \]
for a.e.~$r>0$ such that $B_r(x) \subset \Omega$. For $\mathcal{H}^{n-1}$-a.e.~$x \in J_{u(t)}$, 
\begin{multline} \label{jump_Dusk_ineq} \int_{B_r(x) \cap J} |u^+(t,y)-u^-(t,y)| d\mathcal{H}^{n-1}(y) \leq \int_{B_r(x)} |u^+(t,y)-u^-(t,y)| d\mathcal{H}^{n-1}(y)\\ \leq \int_{B_r(x)} |u^+(s_k,y)-u^-(s_k,y)| d\mathcal{H}^{n-1}(y) \leq \int_{B_r(x)} |Du(s_k)| ,
\end{multline} 
where $J$ is one of the $C^1$ graphs covering $J_{u(t)}$ by rectifiability such that $x \in J$. Dividing both sides of \eqref{jump_Dusk_ineq} by $|B_r^{n-1}|$, we deduce
\begin{equation} \label{jump_Du0_est} |u^+(t,x)-u^-(t,x)| \leq \liminf_{r \to 0^+} \frac{1}{|B_r^{n-1}|} \int_{B_r(x)} |Du_0| \quad \text{for } \mathcal{H}^{n-1}\text{-a.e.}\ x \in J_{u(t)}.
\end{equation} 
Thus, by \cite[Proposition 3.92]{AFP}, we conclude that $J_{u(t)} \subset J_{u_0}$ up to a $\mathcal{H}^{n-1}$-negligible set. Then, using the Radon--Nikodym derivation theorem, we also deduce that $|u^+(t)-u^-(t)| \leq |u_0^+ - u_0^-|$ $\mathcal{H}^{n-1}$-a.e.~on $J_{u(t)}$. Summing up, we have obtained 
\begin{thm} \label{jump_bound} 
Let $u_0 \in BV(\Omega) \cap L^n(\Omega)$. Then for a.e.~$t>0$, $J_{u(t)} \subset J_{u_0}$ up to a $\mathcal{H}^{n-1}$-negligible set and 
\[ |u^+(t)-u^-(t)| \leq |u_0^+ - u_0^-| \quad \mathcal{H}^{n-1}\text{-a.e.~on } J_{u(t)}.\]
\end{thm} 
The thesis of the theorem can also be rephrased as $|D^j u(t)| \leq |D^j u_0|$ in the sense of measures. As we have mentioned before, analogous inequality for the absolutely continuous part of $Du$ fails in general if $n>1$. To our knowledge it remains an open question whether $|D^c u(t)|\leq |D^c u_0|$. 

As for the fourth-order case, we have seen in Section \ref{S4} that the jump inclusion $J_{u(t)} \subset J_{u_0}$ does not hold, as the jumps can move. Moreover, jump discontinuities can emerge out of Lipschitz continuous initial data, even in the 1D case \cite{GG}. 
}

\end{document}